\documentclass[a4paper,10pt]{article}
\usepackage[latin1]{inputenc}
\usepackage[T1]{fontenc}
\usepackage{graphicx}
\usepackage{pstricks,pstricks-add,pst-math,pst-xkey}
\usepackage{amsmath,amssymb}
\usepackage{url}
\usepackage{amsthm}
\usepackage{fancybox}
\usepackage{enumitem}
\usepackage{hyperref}
\usepackage{fancyhdr}
\usepackage[all]{xy}
\usepackage{stmaryrd }

\newtheorem{df}{Definition}[section]
\newtheorem{pr}{Proposition}[section]
\newtheorem{cor}{Corollary}[section]

\newtheorem{lem}{Lemma}[section]

\newtheorem{thm}{Theorem}[section]

\newcommand{\lig}{\mathfrak{g}}

\newcommand{\Hi}{{\cal H}}
\def\ignore#1{}
\newcommand{\application}[5]{\begin{array}{l|rcl}
#1: & #2 & \longrightarrow & #3 \\
    & #4 & \longmapsto & #5 \end{array}}
\title{{\bf Meromorphic continuation approach to noncommutative geometry}}
\author{ Franck GAUTIER-BAUDHUIT\\
                              ~~\\
                              ~~\\
         {\sl Laboratoire de Mathématiques (UMR 6620), Université Blaise Pascal},\\
         {\sl 3 place Vasarély, CS 60026, 63178 Aubière Cedex, France}.\\
                              ~~\\
          \href{mailto:franck.gautier@math.univ-bpclermont.fr}{franck.gautier@math.univ-bpclermont.fr}}
\begin{document}

\maketitle

 \begin{abstract}
Following an idea of Nigel Higson, we develop a method for proving the existence of a meromorphic continuation for some spectral zeta functions. The  method is based on algebras of generalized differential operators.  The main theorem states, under some conditions, the existence of a meromorphic continuation, a localization of the poles in supports of arithmetic sequences and an upper bound of their order. We give an application in relation with a class of nilpotent Lie algebras.
 \end{abstract}

\tableofcontents

\section*{Introduction} Let ${\cal M}$ be a smooth, closed manifold and $\Delta$  a Laplace-type operator on ${\cal M}$. For every linear differential operator $X$ on ${\cal M}$ the zeta function $\zeta_{X,\Delta}:z \mapsto {\rm Trace}(X \Delta^z)$ is holomorphic in some left half-plane within $\mathbb{C}$ and admits a meromorphic continuation to all of $\mathbb{C}$. This result was first observed by Minakshisundaram and Pleijel \cite{zetaMP} in 1949,  and totally proved by Seeley \cite{Seeley} in 1967. In 2006, Nigel Higson suggested in \cite{Nhigson} a new proof of this meromorphic continuation. In the introduction N.Higson  wrote "it would be interesting to see whether or not the basic method can be adapted to other more complicated situations". This suggestion was the reason of the present article.\\

This article is built in two parts. In the first one,  we expose  a general method to study meromorphic continuation for some zeta functions. In the second, we apply this method to a particular family of nilpotent Lie algebras.\\

In  this paper $\Hi$ is  a separable complex Hilbert space,  $\Delta$  is a self-adjoint, positive and invertible linear operator on $\Hi$ with compact resolvent. We denote $\Hi^{\infty}=\bigcap_{n\in \mathbb{N}}{\rm Dom}(\Delta^n)$, it is a dense subspace of $\Hi$ (see \cite[page 3]{Otgon}).\\

In a first part we recall some well known results about Sobolev spaces associated to $\Delta$ and generalized pseudo-differential calculus. Most of the ideas are in the article \cite{Otgon} and book \cite{SimonReed}. Then we set in Paragraph \ref{algdif} a framework for the method, namely the algebras of generalized differential operators. An important tool is Taylor expansions of commutators like $[X,\Delta^z]$  (Lemma $\ref{lem:Taylor}$) in terms of holomorphic families (Definition $\ref{def:holofamily}$). These expansions were first established by A.Connes and H.Moscovici \cite[Appendix B]{Connes1}.  In his article \cite{Nhigson} N.Higson used this formula with some operators like\\

 $$\application {H_{a,b}}{{\cal D}}{{\cal D}}{W}{\displaystyle{\sum_{i=1}^n a_i [-Q_i,P_iW]+\sum_{i=1}^n b_i [P_i,Q_iW]},}$$

 where ${\cal D}$ was an algebra of generalized differential operators, $a_i,b_i$ some complex numbers and $P_i,Q_i$ some operators in  ${\cal D}$. In \cite{Nhigson} the operators $P_i$ and $Q_i$ are subject to the supplementary equality
\begin{equation}
\sum_i[P_i,Q_i]=\rho {\rm Id},
\end{equation}
where $\rho$ is the dimension of the manifold. A similar condition, namely
\begin{equation}
\sum_i(a_i+b_i)[P_i,Q_i]=\rho {\rm Id}
\end{equation}

for some complex number $\rho$, will be required in Section \ref{sect:deux}. To investigate some situations more "noncommutative" than operators on manifold, we need to compose such operators and then formulate some Taylor expansions related to these compositions, see Lemmas $\ref{lem:Taylor Hab X}$, $\ref{lem:Taylor HlambdaX}$ and $\ref{lem:Taylor Hzx}$.\\

  We assume that $\Delta$ is in some Schatten ideal, namely some powers of $\Delta$ are trace class operators. Then   for any integer number $k$,  the holomorphic families of type $k$ define holomorphic zeta functions on left half-plane ${\rm Re}(z)<k\alpha$ with a fixed $\alpha>0$. The  aim is to show that holomorphic families of type $k$ are somehow  eigenvectors modulo holomorphic families of type $k-1$ for operators made of commutators. To achieve this  we proceed in two times. In the first one, we establish a reduction at algebraic level, using the main new tool of this article called  reduction sequence $\eqref{suitereduc}$. In the second one, we extend this reduction to all holomorphic families in the important  Theorem $\ref{lem:Reductionmero}$. Theorem $\ref{thetheorem}$ shows the existence of meromorphic continuation for some zeta functions $\zeta_{X,\Delta}:z \mapsto {\rm Trace}(X \Delta^z)$, and in the same time  gives a localization of the poles and a majoration of their orders.\\

The method consists  in building a  reduction sequence $\eqref{suitereduc}$, which is fundamentally an algebraic problem. For this goal we must have a good knowledge of the algebra of generalized differential operators. For instance in the  manifold's  case we deal with Weyl algebras \cite{Dixmier1}.\\

As an application of the method exposed in the first part, we prove in the second part the meromorphic continuation for zeta functions associated with nilpotent Lie algebras. Let $G$ be a connected and simply connected nilpotent Lie group, let $\lig$ be its Lie algebra and ${\cal U}(\lig)$ its universal enveloping algebra.  In his article \cite[Theorem 7.1]{Kirillov} Kirillov proves that every unitary and irreducible representation $\pi$ of the group $G$ gives by derivation a surjective representation $\rho$ of the universal algebra ${\cal U}(\lig)$ on a Weyl algebra. Given  a basis $(X_1,\dots,X_n)$ of $\lig$, we shall define the Goodman-Laplacian \cite{Goodman} as $\Delta=1+\rho(-X_1^2-\dots-X_n^2)$. It allows us to make the Weyl algebra  ${\cal D}=\rho({\cal U}(\lig))$ an algebra of  generalized differential operators  for which $\rho(\Delta)$ is a generalized laplacian of order two, and to define the zeta functions $\zeta_{X,\Delta}:z \mapsto {\rm Trace}(X \Delta^z)$ where $X$ is in ${\cal D}$. The kernel  of $\rho$ is a nontrivial two-sided ideal of ${\cal U}(\lig)$.\\

Let $n\in \mathbb{N}^*$ and consider a $n$-uple $\alpha \in (\mathbb{N}^*)^n$. Denote by ${\cal I}$ a partition of $\llbracket 1,n \rrbracket$. We associate to $\alpha$ and ${\cal I}$ a nilpotent Lie algebra  $\lig_{\alpha,{\cal I}}$ in Paragraph \ref{liealpha}. The center of ${\cal U}(\lig_{\alpha,{\cal I}})$ is a one dimensional vector space generated by the element $Y^{(0,\dots,0)}$. We consider the representation $\rho$ of ${\cal U}(\lig_{\alpha,{\cal I}})$ associated to the linear form $f=(Y^{(0,\dots,0)})^*$ by the orbit method of Kirillov \cite{Kirillov}. In this context $\rho(\Delta)$ is a Schrödinger operator.
Giving an operator in ${\cal D}$, we have to find a preimage with minimal degree in ${\cal U}(\lig)$. For this we use the theory of Gröbner bases for ideals, but no knowledge outside Proposition $\ref{pr:OI}$ is necessary in this paper. The generators of the ideal ${\rm Ker}(\rho)$ exhibited by C. Godfrey \cite{Godfrey} are recalled in Proposition $\ref{lemgenIf}$, and after some computations we obtain a Gröbner basis of ${\rm Ker}(\rho)$  in Corollary $\ref{corgenIf}$. The Gr\"obner basis provides one particular preimage (under $\rho$) of any element of $\mathcal D$ with minimal degree, and gives useful information on the dominant term of this preimage. This information is essential for proving some reduction results in the algebra ${\cal U}(\lig_{\alpha,{\cal I}})$ (see Lemma $\ref{redOif}$ and Proposition $\ref{Ld}$), and for constructing a reduction sequence $(T_s)_{s\in \mathbb{N}}$ $\eqref{TSLIE}$. Finally we prove the meromorphic continuation of the zeta functions $\zeta_{X,\Delta}$ in the context of the nilpotent Lie algebras $\mathfrak g_{\alpha,\mathcal I}$ (Theorem $\ref{th:merolig}$) as an application of Theorem \ref{thetheorem}.

\bigskip

\noindent {\bf Acknowledgements}: I thank Jean-Marie Lescure and Dominique Manchon for their valuable help.

\section{Meromorphic continuation via reduction sequences}
  In this section $\Hi$ is  a separable complex Hilbert space,  $\Delta$  is a self adjoint, positive and invertible linear operator on $\Hi$ with compact resolvent.  The spectrum of $\Delta$  is a  nondecreasing sequence $0<\lambda_0 \le \lambda_1\le \lambda_2 \le \dots$ of real numbers which diverges to $+\infty$. We denote by  $\Gamma$  a downwards pointing vertical line in $\mathbb{C}$ separating 0 from the spectrum of $\Delta$.

\subsection{An algebra of operators}

$\label{Sobolev}$

\begin{df} Let   ${\cal A}$ be an associative algebra. Let $I=\mathbb{N},\mathbb{Z}$ or $\mathbb{R}$. A family $({\cal A}^i)_{i\in I}$  composed of subsets of ${\cal A}$ is an increasing filtration of ${\cal A}$ if the following conditions are satisfied:
\begin{enumerate}[nolistsep]
\item For all  $i,j \in I$: $i\le j \Rightarrow {\cal A}^i \subset {\cal A}^j$.
\item  $\displaystyle{{\cal A}=\cup_{i\in I} {\cal A}^i}$.
\item For all $i,j \in I$: ${\cal A}^i{\cal A}^j \subset {\cal A}^{i+j}$.
\end{enumerate}
An element $a\in {\cal A}$ has order  less than or equal to $i$ if $a \in {\cal A}^i$.
\end{df}
We shall write "Let $a\in {\cal A}^{i_a}$" instead of "Let $a\in {\cal A}$ with order  less than or equal to $i_a$".
\ignore{Now let ${\cal A}$ be an associative algebra with an increasing filtration $({\cal A}^i)_{i\in I}$. Let ${\cal B}$ a subalgebra of ${\cal A}$. The family $({\cal B} \cap{\cal A}_i)_{i\in I}$ is an increasing filtration of ${\cal B}$ called induced filtration on ${\cal B}$ by $({\cal A}^i)_{i\in I}$. The subalgebra ${\cal B}$ is named filtered subalgebra of ${\cal A}$ if exist an increasing filtration  $({\cal B}^i)_{i\in I}$ of $B$ s.t for any $i\in I$, we have ${\cal B}^i \subset {\cal A}^i$.}

\begin{df} Denote by $\Hi^{\infty}$ the space of vectors common to the domains of all powers of $\Delta$:
\begin{equation*} \displaystyle{\Hi^{\infty}=\bigcap_{n\in \mathbb{N}}{\rm Dom}(\Delta^n) \subseteq  \Hi },\end{equation*}
where ${\rm Dom}(\Delta^n)$ is the domain of $\Delta^n$.
\end{df}

 The space $\mathcal H_\infty$ is dense in $\Hi$ (see \cite[page 3]{Otgon}).  We denote by ${\rm End}(\Hi^{\infty})$  the algebra of linear endomorphisms of $\Hi^{\infty}$. Let $T$ be an unbounded operator with $\Hi^{\infty}\subseteq {\rm Dom}(T)$, if $\Hi^{\infty}$ is stable for $T$ then $T_{|\Hi^{\infty}} \in {\rm End}(\Hi^{\infty})$, we shall still denote by $T$ instead of $T_{|\Hi^{\infty}}$ this endomorphism.

\begin{df} $\label{algdif}$
 Let $\Hi$ be  a separable complex Hilbert space  and $\Delta$   a self adjoint, positive and invertible linear operator on $\Hi$ with compact resolvent. Let ${\cal D}$  be a subalgebra of  ${\rm End}(\Hi^{\infty})$ and $r \in \mathbb{N}^*$. We shall call $\left ({\cal D},\Delta,r\right) $ an {\sl algebra of generalized differential operators} associated with the  generalized Laplace operator $\Delta$ of order  $r$, if the following conditions are satisfied:
\begin{enumerate}[nolistsep]
\item The algebra ${\cal D}$ is filtered with an increasing filtration  $({\cal D}^q)_{q \in \mathbb{N}}$.
\item The unit of ${\cal D}$ (if it  exists) has an order equal to 0.
\item For every  $X\in {\cal D}^{q_X}$,  $[\Delta,X]\in  {\cal D}^{q_X+r-1}$.
\item $\label {ellip}$ For all  $X\in {\cal D}^{q_X}$, there exists $\varepsilon>0 $ s.t:
 \begin{equation*} \forall v \in \Hi^{\infty},~~\|\Delta^{\frac{q_X}{r}}v \|+\|v\|\ge \varepsilon \|Xv\|~~~\text{~(Generalized G{\aa}rding inequality).}\end{equation*}
\end{enumerate}
\end{df}

\begin{df} $\label{Sobolev}$ Let $r\in \mathbb{N}^*$. For every $s\in \mathbb{R}$,  $\Hi^s$ is the completion in  $\Hi$ of ${\rm Dom}(\Delta^{\frac{s}{r}})$ for the norm $\|.\|_s$ associated with the scalar product:
\begin{equation*} \forall u,v \in {\rm Dom}(\Delta^{\frac{s}{r}}),~~(u|v)_s=(\Delta^{\frac{s}{r}}u,\Delta^{\frac{s}{r}}v)_{ \Hi}.\end{equation*}
\end{df}

As a consequence of the above definition, $\Delta^{\frac{m}{r}}: \Hi^s\rightarrow \Hi^{s-m}$ is a unitary operator for any real numbers  $s$ and $m$.

\begin{lem} $\label{lem:Hst}$ For every $s \ge t$, the embedding $\Hi^s \subset \Hi^t$ is continuous. Moreover, for every $s\ge 0$, ${\rm Dom}(\Delta^{\frac{s}{r}})$ is a complete space and therefore $\Hi^s={\rm Dom}(\Delta^{\frac{s}{r}})$.
\end{lem}

 The function $s \mapsto \Hi^s$ is decreasing with respect to the inclusion, so $  \displaystyle{\Hi^{\infty}=\bigcap_{s\in \mathbb{R}}\Hi^{s}}$. For every $z\in \mathbb{C}$, the operators $\Delta^z$ defined by functional calculus belongs to ${\rm End}(\Hi^{\infty})$\label{:andelta}.
 The subspace $\Hi^{\infty}$ is dense in  ${\rm Dom}(\Delta^z)$ for all $z\in \mathbb{C}$, in particular $\Hi^{\infty}$ is dense in $\Hi^s$ for every $s\in \mathbb{R}$.   For any $z\in \mathbb{C}$ and $s\in \mathbb{R}$, the operator $\Delta_{|\Hi^{\infty}}^z:\Hi^{\infty}\rightarrow  \Hi^{\infty}$  extends to a unitary operator $\Delta^z:\Hi^{s+r{\rm Re}(z)}(\Delta)\rightarrow \Hi^s(\Delta)$.

\begin{df}

For every $t\in \mathbb{R}$, denote by ${\rm Op}^t$  the subspace of ${\rm End}(\Hi^{\infty})$ made of operators which admit a continuous extension from $\Hi^{s+t}$ to $\Hi^s$ for every $s\in \mathbb{R}$ and let
\begin{equation*}  {\rm Op}:=\bigcup_{t\in \mathbb{R}}{\rm Op}^t~~{\text and}~~{\rm Op}^{-\infty}:=\bigcap_{t\in \mathbb{R}}{\rm Op}^t. \end{equation*}
\end{df}

An operator in ${\rm Op}^{-\infty}$ is in particular bounded on $\Hi^0$ and maps ${\cal H}^0$ into $\Hi^{\infty}$. The family $(Op^t)_{t\in \mathbb{R}}$ is an increasing filtration of the algebra ${\rm Op}$. If follows that ${\rm Op}^0$ is a subalgebra of ${\rm Op}$ and ${\rm Op}^{-\infty} \subseteq {\rm Op}$. For every $t\le 0$, ${\rm Op}^t$ is a two-sided ideal of  ${\rm Op}^0$.

\begin{df} $\label{lem:Opalg}$
Let  $T\in {\rm Op}$ and $t\in \mathbb{R}$. We say that $T$ has an    analytic order  less than or equal to $t$ if $T \in Op^t$.
\end{df}

\begin{pr} Let $\Delta$  be a self adjoint, positive and invertible linear operator on $\Hi$ with compact resolvent. Then,
 \begin{equation*}\label{analdeltaz} \forall  z\in \mathbb{C},~~ \Delta^{z} \in {\rm Op}^{r{\rm Re}(z)}.\end{equation*}
\end{pr}

\begin{proof} For every $s\in \mathbb{R}$ and $z\in \mathbb{C}$, the spectral calculus proves that $\Delta^{i{\rm Im}(z)}$ is a unitary operator on $\Hi^s$.  Let be $u \in \Hi^{\infty}$, then
$$
\|\Delta^z u\|_s=\|\Delta^{i{\rm Im}(z)}\Delta^{{\rm Re}(z)}u\|_s
=\|\Delta^{{\rm Re}(z)}u\|_s
=\|u\|_{s+r{\rm Re}(z)}.
$$
\end{proof}

 \begin{lem} $\label{cor:Op}$ Let $\Delta$ be  a self adjoint, positive and invertible linear operator on $\Hi$ with compact resolvent. For every $t\in \mathbb{R}$,
\begin{equation*} {\rm Op}^t=\Delta^{\frac{t}{r}}{\rm Op}^0={\rm Op}^0\Delta^{\frac{t}{r}}.\end{equation*}
\end{lem}
\begin{proof} We prove the first equality, the second admits a similar proof. For any $t\in \mathbb{R}$,
$$Op^t=\Delta^0Op^t=\Delta^{\frac{t}{r}}\Delta^{-\frac{t}{r}}Op^t \subseteq \Delta^{\frac{t}{r}}Op^{-t}Op^t \subseteq \Delta^{\frac{t}{r}}Op^0\subseteq Op^tOp^0 \subseteq Op^t.$$
\end{proof}

\begin{cor}\label{prop:OpL}
Let $p\in \mathbb{R}$ and denote by ${\cal L}^p(\Hi)$ the ideal of  $p$-Schatten operators. If $\Delta^{-\frac{1}{r}}\in {\cal L}^p(\Hi)$ with $p\ge1$, then for every  $0<t\le p$, we have ${\rm Op}^{-t} \subset {\cal L}^{\frac{p}{t}}(\Hi)$.
\end{cor}
\begin{proof}
From Lemma  $\ref{cor:Op}$ we have for every $t\in \mathbb{R}$, ${\rm Op}^{-t}= \Delta^{-\frac{t}{r}}{\rm Op}^0\subset\Delta^{-\frac{t}{r}}{\cal B}(\Hi)$. Let $p\ge 1$,  the inclusions follow since ${\cal L}^p(\Hi)$ is a two-sided ideal of ${\cal B}(\Hi)$.
\end{proof}
\begin{lem}  Any  algebra of generalized differential operators $\left ({\cal D},\Delta,r\right)$
 is a filtered subalgebra of $Op$, that is for every $q\in \mathbb{N}$ we have:
\begin{equation*}\label{DqincluOpq} {\cal D}^q\subseteq {\rm Op}^q. \end{equation*}
\end{lem}
\begin{proof} Let $q \in \mathbb{N}$, as $s\le q+s$  the inclusion from $\Hi^{s+q}$ in $\Hi^s$ is continuous (Lemma $\ref{lem:Hst}$), and $\Delta^{\frac{q}{r}}$ is bounded from $H^{s+q}$ to $H^s$. Hence the lemma is a consequence of the generalized G{\aa}rding inequality $\eqref{ellip}$.
\end{proof}

 Let $\left ({\cal D},\Delta,r\right) $ be an algebra of generalized differential operators, denote by ${\cal D}[\Delta]$  the vector space of polynomials in $\Delta$ with coefficients in ${\cal D}$.

\begin{lem} Let $\left ({\cal D},\Delta,r\right) $ be an algebra of generalized differential operators and
 let $X\in {\rm End}(\Hi^{\infty})$.  For every $k\in \mathbb{N}$ we denote by $X^{(k)}$ the $k$-fold commutator ${\rm ad}^k\Delta(X)$. For every $k\in \mathbb{N}$ , if $X \in {\cal D}^{q_X}$ then $X^{(k)}\in {\cal D}^{q_X+k(r-1)}$.
\end{lem}
\begin{proof} The algebra ${\cal D}$ is invariant under derivation $[\Delta,.]$, so the operator $X \mapsto X^{(k)}$ is a linear endomorphism of ${\cal D}$, respectively ${\cal D}[\Delta]$. As $\Delta$ is a generalized laplacian  of order  $r$, by definition we have for every $X \in {\cal D}^{q_X}$,  $X^{(1)}\in {\cal D}^{q_X+r-1}$. The result about the order of $X^{(k)}$ follows by induction.
\end{proof}

\begin{lem}$\label{lemDeltaalg} $
 Let $\left ({\cal D},\Delta,r\right) $ be an algebra of generalized differential operators.
For every $X\in {\cal D}$ and every $i\in \mathbb{N}^*$ we have the exact formula
  \begin{equation*} [\Delta^i,X]= \sum_{k=1}^i \binom{i}{k}X^{(k)}\Delta^{i-k}. \end{equation*}

 \end{lem}
\begin{proof} We proceed by induction on $i$. For $i=1$, $\Delta X- X \Delta=[\Delta,X]$. Now,
\begin{gather*}
\begin{split}
[\Delta^{i+1},X]&=[\Delta^i,[\Delta,X]]+[\Delta^i,X]\Delta+[\Delta,X]\Delta^i \\[-6pt]
                &=\sum_{k=1}^i \binom{i}{k}X^{(k+1)}\Delta^{i-k}+\sum_{k=1}^i \binom{i}{k}X^{(k)}\Delta^{i+1-k}+[\Delta,X]\Delta^i\\[-6pt]
                &=X^{(i+1)}+\sum_{k=2}^i (\binom{i}{k-1}+\binom{i}{k})X^{(k)}\Delta^{i+1-k}+iX^{(1)}\Delta^{i}+[\Delta,X]\Delta^i\\[-6pt]
                &=\sum_{k=1}^{i+1} \binom{i+1}{k}X^{(k)}\Delta^{i+1-k}.
\end{split}
\end{gather*}
\end{proof}

For every $X\in {\cal D}$ and for every $i\in \mathbb{N}$, $\Delta^iX \in {\cal D}[\Delta]$,  since $\Delta^iX= X \Delta^i+\sum_{k=1}^i \binom{i}{k}X^{(k)}\Delta^{i-k}$. Hence ${\cal D}[\Delta]$ is a subalgebra of $Op$.\\

We now define for ${\cal D}[\Delta]$ a natural structure of algebra of generalized differential operators for which  $\Delta$ will be  a generalized Laplace operator  of order $r$, and the algebra $\left ({\cal D},\Delta,r\right) $  will be a filtered subalgebra of $\left ({\cal D}[\Delta],\Delta,r\right) $. Let $q\in \mathbb{N}$, we define
\begin{equation*} {\cal D}[\Delta]^q=\sum_{\begin{matrix} q'+rq''=q \cr q',q'' \in \mathbb{N} \end{matrix}} {\cal D}_{q''}^{q'}[\Delta]. \end{equation*}
 As usual, for an algebra ${\cal A}$, ${\cal A}_p[X] \subset {\cal A}[X]$ is the   subspace  of polynomials with degree less than or equal to  $p$ and with coefficients in ${\cal A}$.

\begin{lem} Let $\left ({\cal D},\Delta,r\right) $ be an algebra of generalized differential operators.  The algebra ${\cal D}(\Delta)$ is filtered by $({\cal D}[\Delta]^q)_{q\in \mathbb{N}}$.
\end{lem}

\begin{proof}
 The inclusions ${\cal D}[\Delta]^{q_1} \subset {\cal D}[\Delta]^{q_2}$ if $q_1 \le q_2$ and the equality $\displaystyle{\cup_{q \in \mathbb{N}} {\cal D}[\Delta]^{q}={\cal D}[\Delta]}$ are immediate. Let $q,q' \in \mathbb{N}$, $X\in {\cal D}^{q_X}$ and $Y \in {\cal D}^{q_Y}$, we suppose $X \Delta^i \in {\cal D}[\Delta] ^q$ and $Y\Delta^j \in {\cal D}[\Delta]^{q'}$. For every $k\in \mathbb{N}$, $Y^{(k)} \in{\cal D}^{q_Y+k(r-1)}$ and for every $k\in \llbracket 0,j \rrbracket$,
\begin{equation*} XY^{(k)}\Delta^{i+j-k} \in {\cal D}_{i+j-k}^{q_X+q_Y+k(r-1)}[\Delta]\subset {\cal D}[\Delta]^{q_X+ri+q_Y+rj-k}\subset {\cal D}[\Delta]^{q+q'}. \end{equation*}

The product
$(X\Delta^i)(Y\Delta^j)=XY \Delta^{i+j}+\sum_{k=1}^j \binom{j}{k}XY^{(k)}\Delta^{i+j-k}$ is then an element ${\cal D}[\Delta]^{q+q'}$ and the inclusion
${\cal D}[\Delta]^{q}{\cal D}[\Delta]^{q'}\subset {\cal D}[\Delta]^{q+q'}$ follows.
\end{proof}

\begin{pr}  Let $\left ({\cal D},\Delta,r\right) $ be an algebra of generalized differential operators.  The triple $\left ({\cal D}[\Delta],\Delta,r\right) $ is also an algebra of generalized differential operators.
\end{pr}

\begin{proof}   From Lemma $\ref{lemDeltaalg} $, ${\cal D}[\Delta]$ is already a subalgebra of $Op$.
We prove each point of Definition $\ref{algdif}$.
\begin{enumerate}[nolistsep]
\item The algebra ${\cal D}[\Delta]$ is unitary if ${\cal D}$ is and in this case $1 \in {\cal D}[\Delta]^{0} $.
\item Let $X \Delta^i \in {\cal D}[\Delta]^{q}$ with $X\in {\cal D}^{q_X}$,  then $[\Delta,X\Delta^i]=[\Delta,X]\Delta^i \in {\cal D}[\Delta]^{q_X+r-1+ri}={\cal D}[\Delta]^{q+r-1}$.
\item Proof of generalized G{\aa}rding inequality: let $X \Delta^i \in {\cal D}[\Delta]$ with $order(X)\le q$. We establish the existence of $\varepsilon >0$ such that for any $v \in {\cal H}^{\infty}$ one has
\begin{equation}
\|\Delta^{\frac{q+ri}{r}}v \|+\|v\|\ge \varepsilon \|X\Delta^i v\|.
\end{equation}
There exists $\varepsilon>0 $ such that for every $v \in \Hi^{\infty}(\Delta)$ we have $\|\Delta^{\frac{q+ri}{r}}v \|+\|\Delta^iv\|\ge \varepsilon \|X\Delta^i v\|$ by applying $\eqref{ellip}$  to $\Delta^iv \in \Hi^{\infty}$. If $i=0$, the inequality holds. Otherwise $\frac{q+ri}{r} \ge i$ and the embedding $\Hi^{\frac{q+ri}{r}} \subset \Hi^{i}$ is continuous, so  there exists $K>0$ such that for every $v \in \Hi^{\infty}(\Delta)$ we have $\|\Delta^iv\|\le K\|\Delta^{\frac{q+ri}{r}}v\|$. Then for every $v \in \Hi^{\infty}(\Delta)$ we obtain
     $\|\Delta^{\frac{q+ri}{r}}v \|\ge \frac{\varepsilon}{1+K} \|X\Delta^i v\|$, hence $\|\Delta^{\frac{q+ri}{r}}v \|+\|v\|\ge \frac{\varepsilon}{1+K} \|X\Delta^i v\|$.
\end{enumerate}

Hence  $\Delta$ is a generalized Laplace operator of order $r$.
\end{proof}

 If $\Delta$ is an element of ${\cal D}$ then the two algebras ${\cal D}$ and ${\cal D}[\Delta]$ are the same as well as the two filtrations.

\subsection{Cauchy's formula}

 For every $z\in \mathbb{C}$, we define by spectral calculus the operator $\Delta^z$ on $\Hi$ with domain containing $\Hi^{\infty}$. Recall that $\Gamma$ is a downwards pointing vertical line in $\mathbb{C}$ separating 0 from the spectrum of $\Delta$. For every $ z \in \mathbb{C}$ such that ${\rm Re} (z)<0$, we have  Cauchy's formula
\begin{equation*} \label{Cauchyformula}\Delta^{z}=\frac{1}{2i\pi}\int_{\Gamma} \lambda^z(\lambda-\Delta)^{-1} d\lambda. \end{equation*}

 This  generalized Riemann integral is convergent for the $Op^0$ operators topology (i.e it is convergent for the norm operators from $\Hi^{s}$ to $\Hi^s$, for every $s\in \mathbb{R}$). For every  $p\in \mathbb{N}$ and every  $ z \in \mathbb{C}$ such that ${\rm Re} (z)<p$, we have the formula

\begin{equation*} \binom {z}{p} \Delta^{z-p}=\frac{1}{2\pi i}\int_{\Gamma} \lambda^{z}(\lambda-\Delta)^{-p-1}d\lambda.\end{equation*}

 Here $\displaystyle{\binom {z}{p}}$ is the usual binomial coefficient.  The integral is convergent for the  $Op^p$ operators topology (namely it is convergent for the bounded norm operators from $\Hi^{s+p}$ to $\Hi^s$, for every $s\in \mathbb{R}$).

$\label{traceholo}$

  Let $a\in \mathbb{R}$, denote
\begin{equation*} \mathbb{C}_{>a}=\{ z\in \mathbb{C}|~{\rm Re}(z)>a\} ~{\text and}~\mathbb{C}_{<a}=\{ z\in \mathbb{C}|~{\rm Re}(z)<a\}. \end{equation*}

\begin{pr}$\label{Deltaholo}$
The function $z\mapsto \Delta^{-z}$ defined from the complex open half-plane $\mathbb{C}_{>0}$ to ${\cal B}(\Hi)$ is holomorphic.
\end{pr}

\begin{proof}
  Let $z_0\in \mathbb{C}_{>0}$. We consider $\Gamma:=\{c+iy | y \in \mathbb{R}\}$ with fixed $c \in ]0,\min(\lambda_0,{\rm Re}(z_0))[$.

  From Cauchy's formula \eqref{Cauchyformula}, we have for every $z \in \mathbb{C}_{>c}$,  $\Delta^{-z}=\frac{1}{2i\pi}\int_{\Gamma}\lambda^{-z}(\Delta-\lambda)^{-1}d\lambda$. The function $z\mapsto \lambda^{-z}$ is differentiable on $\mathbb{C}_{>c}$ with $(\lambda^{-z})'=-\log(\lambda)\lambda^{-z}$. The Bertrand   integral $\frac{1}{2i\pi}\int_{\Gamma}-\log(\lambda)\lambda^{-z}(\Delta-\lambda)^{-1}d\lambda$ is normally convergent on $\mathbb{C}_{>c}$. We can differentiate under the integral symbol, so the application  $z\mapsto \Delta^{-z}$ is differentiable on $\mathbb{C}_{>c}$, in particular at $z_0$.
\end{proof}

\begin{pr}$\label{DeltaXholo}$ Let $X \in {\rm Op}$. If for some real $a$,  $X\Delta^{-a}$ is a trace class operator then the function $z\mapsto {\rm Trace}(X\Delta^{-z})$ is holomorphic on $\mathbb{C}_{>a}$.
\end{pr}

\begin{proof} Let $\zeta:z \mapsto {\rm Trace}(X\Delta^{-z})$. For ${\rm Re}(z)>a$, from spectral calculus we obtain the Taylor expansion
\begin{equation*} \Delta^{-z-h}=\Delta^{-z}-h\ln(\Delta)\Delta^{-z}+\Delta^{-z}\underset{h\rightarrow 0}o(h), \end{equation*}
 where $\underset{h\rightarrow 0}o(h)$ represents bounded operators on $\Hi$ such that $\underset{h\rightarrow 0} \lim \|h\|=0$.

\begin{equation*}
\begin{split}
{\rm Trace}(X\Delta^{-z-h})-{\rm Trace} (X\Delta^{-z})&={\rm Trace}(X\Delta^{-z-h}-X\Delta^{-z})\\
                                                           &=h{\rm Trace}(X\ln(\Delta)\Delta^{-z})+{\rm Trace}(X\Delta^{-z}\underset{h\rightarrow 0}o(h))\\
                                                           &=h{\rm Trace}(X\ln(\Delta)\Delta^{-z})+\underset{h\rightarrow 0}o(h).
\end{split}
\end{equation*}
 The last $\underset{h\rightarrow 0}o(h)$ is a consequence of the following general result:
let $T\in {\cal L}^1(\Hi)$ and $S\in {\cal B}(\Hi)$,  the inequality
$|{\rm Trace}(TS)|\le |{\rm Trace}(T)|\|S\|$ holds.  Then $\left | {\rm Trace}(X\Delta^{-z})\underset{h\rightarrow 0}o(h) \right |\le \left | {\rm Trace}(X\Delta^{-z})\right |\|\underset{h\rightarrow 0}o(h)\|$.

  Hence the application $\zeta:z \mapsto {\rm Trace}(X\Delta^{-z})$ is holomorphic on $\mathbb{C}_{>a}$, and we shall precise that for every $z\in \mathbb{C}_{>a}$
\begin{equation*} \zeta'(z)={\rm Trace}(X\ln(\Delta)\Delta^{-z}). \end{equation*}
\end{proof}

\begin{pr}$\label{holozeta1}$ If the operator $\Delta^{-\frac{1}{r}}$ is an element of ${\cal L}^p(\Hi)$ for some $p\ge 1$ and if $X\in Op^{-t}$ for some real $t$, then the spectral zeta function $z\mapsto {\rm Trace}(X\Delta^{-z})$ is holomorphic on the half-plane
$\mathbb{C}_{>\frac{p-t}{r}}$.
\end{pr}
\begin{proof} Let $z \in \mathbb{C}$. The operator $X\Delta^{-z}$ is in  $Op^{-t-r{\rm Re}(z)}$, and for $0<t+r{\rm Re}(z)\le p$ we have $Op^{-t-r{\rm Re}(z)}(\Delta)\subset {\cal L}^{\frac{p}{t+r{\rm Re}(z)}}(\Hi)$. In particular if ${\rm Re}(z)=\frac{p-t}{r}$ we obtain $X\Delta^{-z} \in {\cal L}^1(\Hi)$. We conclude with Proposition \ref{DeltaXholo}.
\end{proof}
\noindent The next result will be used in the proof of Theorem  $\ref{thetheorem}$.
\begin{pr}\label{comutnulle}
Let $(X,Y,T)\in (Op)^3$, then $z \mapsto{\rm Trace}([X,Y\Delta^{-z}T])$ vanishes if ${\rm Re}(z)\gg 0$. \end{pr}
\begin{proof} For every $N \in \mathbb{N}$ and $z\in \mathbb{C}$ we write:

 $$[X,Y\Delta^{-z}T]=[XY\Delta^{-N},\Delta^{-z+N}T]+[\Delta^{-z+N}TX,Y\Delta^{-N}].$$

   For a choice of large enough integer $N$, the operators $XY\Delta^{-N}$ and $Y\Delta^{-N}$ are bounded. With this choice of $N$, for complex numbers with  large enough real part, the operators $\Delta^{-z+N}T$ and  $\Delta^{-z+N}TX$  are trace class. The conclusion follows.
\end{proof}
By analytic continuation, we shall say that the function $z \mapsto{\rm Trace}([X,Y\Delta^{-z}T])$ is identically equal to zero on $\mathbb{C}$.
\begin{df}$\label{def:holofamily}$ Let $\left ({\cal D},\Delta,r\right)$ be an algebra of generalized differential operators. Let $k\in \mathbb{Z}$.
\begin{enumerate}[nolistsep]

\item The family $(T_z)_{z \in \mathbb{C}}$   is an elementary holomorphic family of type $k$  if there exists a polynomial function $p$, an integer $n\in \mathbb{N}$ and an operator $X \in {\cal D}^{rn+k}$  such that for every $z\in \mathbb{C}$, $T_z=p(z)X\Delta^{z-n}$.

\item The family  $(T_z)_{z \in \mathbb{C}}$    is a  holomorphic family of type $k$ if for any real number $a$ and $m$,  there exists an integer $q$ and $q$ elementary holomorphic families of type $k$ denoted $(T_z^j)_{z \in \mathbb{C}}$ for $1 \le j \le q$, such that on the half-plane  ${\rm Re}(z)<a$ we have the following  decomposition called Taylor expansion of order $m$ and of frontier $a$:

\begin{equation*} T_z=T_z^1+\dots+T_z^q+R_z, \end{equation*}

 where the remainder  $z\mapsto R_z$  is holomorphic from the half-plane ${\rm Re}(z)<a$ to $Op^m$.
\end{enumerate}
\end{df}

If the family $(T_z)_{z \in \mathbb{C}}$   is a   holomorphic family of type $k$ then $T_z \in Op^{k+r{\rm Re}(z)}$ for every $z\in \mathbb{C}$. A  holomorphic family of type $k$ is also a  holomorphic family of type $k+1$. Sometimes we shall write $ (T_z)_{z \in \mathbb{C}}\in{\rm Hol}(k)$ instead of $(T_z)_{z \in \mathbb{C}}$ is a  holomorphic family of type $k$.  Contrarily to elementary holomorphic families of type $k$, the  holomorphic families of type $k$ are stable under product and commutators with an operator of ${\cal D}$, as we shall see later.

\begin{pr} \label{holozeta2}
Let $\left ({\cal D},\Delta,r\right)$ be an algebra of  generalized differential operators. Let $k\in\mathbb{Z}$ and $(T_z)_{z \in \mathbb{C}}$   be a holomorphic family of type $k$. If $\Delta^{-\frac{1}{r}}\in {\cal L}^p(\Hi)$ for some $p\ge 1$,  then the  function $z\mapsto {\rm Trace}(T_z)$ is holomorphic on the open half-plane $\mathbb{C}_{<\frac{-p-k}{r}}$.
\end{pr}
\begin{proof} Let $T_z=T_z^1+\dots+T_z^q+R_z$ be a Taylor expansion  of order $-p$ and of frontier $a=\frac{-p-k}{r}$ of the family $(T_z)_{z \in \mathbb{C}}$. Every $(T_z^j)_{z \in \mathbb{C}}$ for $1 \le j \le q$ is  an elementary holomorphic family of type $k$, so there exists an integer $n_j$ and an operator  $X_j \in {\cal D}^{rn+k}$ such that $T_z^j=X_j\Delta^{z-n_j}$. By Proposition $\ref{holozeta1}$ the map $z\mapsto {\rm Trace}(X_j\Delta^{z-n_j})$ is holomorphic on $\mathbb{C}_{<a}$. By definition, the function defined on $\mathbb{C}_{<a}$ by $z\mapsto R_z$ takes values  in $Op^{-p} \subset {\cal L}^1(\Hi)$ and is holomorphic. As a consequence $z\mapsto {\rm Trace}(R_z)$ is holomorphic on the open half-plane $\mathbb{C}_{<a}$.
\end{proof}

\subsection{ Taylor's expansion }

The next lemma is due to  A.Connes and H.Moscovici \cite[Appendix B]{Connes1}. We shall refer to it as Taylor's expansion, as the formula agrees with the  Definition \ref{def:holofamily}.

\begin{lem}$\label{lem:Taylor}$ Let $\left ({\cal D},\Delta,r\right)$ be an algebra of generalized differential operators. Let $X\in {\cal D}^{q_X}$ and $z\in \mathbb{C}$. For every integer $N$ such that $N+1>{\rm Re}(z)$ the following identity holds:
\begin{equation*}
 [X,\Delta^{z}]=\displaystyle{-\sum_{k=1}^N} \binom{z}{k}X^{(k)}\Delta^{z-k}+R_N(z,X),\label{devTaylor}
\end{equation*}

where $R_N(z,X)=\displaystyle{-\frac{1}{2\pi i }\int_{\Gamma} \lambda^{z}(\lambda-\Delta)^{-1}X^{(N+1)}(\lambda-\Delta)^{-N-1}d\lambda}.$ For every $k\in \llbracket1,N\rrbracket$, we have $X^{(k)}\Delta^{z-k}\in Op^{r{\rm Re}(z)+q_X-k}$ and $R_N(z,X)\in Op^{r{\rm Re}(z)+q_X-N-1}$.

\end{lem}

\begin{proof}

The eigenvalues of  $\Delta$ consist in a nondecreasing sequence $0<\lambda_0 \le \lambda_1\le \dots \rightarrow +\infty$. We choose $\Gamma$ to be the vertical line given by the equation $x=c$ with $0<c<\lambda_0$. We first suppose ${\rm Re(z)}<0$. For every integer $k$ less than or equal to  $N$, with the Cauchy formula we obtain
$$\binom{z}{k}X^{(k)}\Delta^{z-k}=\frac{1}{2\pi i}\int_{\Gamma} \lambda^{z}X^{(k)}(\lambda-\Delta)^{-k-1}d\lambda.$$
A simple computation yields
 $$X^{(k)}(\lambda-\Delta)^{-1}=(\lambda-\Delta)^{-1}X^{(k)}-(\lambda-\Delta)^{-1}X^{(k+1)}(\lambda-\Delta)^{-1},$$
Hence
$$\binom{z}{k}X^{(k)}\Delta^{z-k}=\frac{1}{2\pi i}\int \lambda^{z}[(\lambda-\Delta)^{-1}X^{(k)}(\lambda-\Delta)^{-k}-(\lambda-\Delta)^{-1}X^{(k+1)}(\lambda-\Delta)^{-k-1}]d\lambda .$$
Summing from $k=0$ to $N$, we get:
\begin{equation*}
{\sum_{k=0}^N} \binom{z}{k}X^{(k)}\Delta^{z-k}=\displaystyle{\frac{1}{2\pi i}}\int \lambda^{z}(\lambda-\Delta)^{-1}X^{(0)}d\lambda-\displaystyle{\frac{1}{2\pi i}}\int\lambda^{z}(\lambda-\Delta)^{-1}X^{(N+1)}(\lambda-\Delta)^{-N-1}d\lambda.
\end{equation*}
It follows that $[X,\Delta^{z}]=\displaystyle{\sum_{k=1}^N} \binom{z}{k}X^{(k)}\Delta^{z-k}-\frac{1}{2\pi i}\int \lambda^{z}(\lambda-\Delta)^{-1}X^{(N+1)}(\lambda-\Delta)^{-N-1}d\lambda$.\\

Let $R_N(z,X)=\displaystyle{-\frac{1}{2\pi i }\int_{\Gamma} \lambda^{z}(\lambda-\Delta)^{-1}X^{(N+1)}(\lambda-\Delta)^{-N-1}d\lambda}$. This integral is convergent for the
$Op^{q_X-N-1}$  topology, hence it defines a holomorphic function on the open half-plane $\mathbb{C}_{<N+1}$. As $z\mapsto [X,\Delta^{z}]=\displaystyle{\sum_{k=1}^N} \binom{z}{k}X^{(k)}\Delta^{z-k}$ is holomorphic on $\mathbb{C}_{<0}$, the formula is available on $\mathbb{C}_{<N+1}$ by analytic continuation. For any $k\in \llbracket1,N\rrbracket$, we have $X^{(k)}\in {\cal D}^{q_X+k(r-1)} \subset Op^{q_X+k(r-1)}$ and $\Delta^{z-k}\in Op^{r({\rm Re}(z)-k)}$, hence $X^{(k)}\Delta^{z-k}\in Op^{r{\rm Re}(z)+q_X-k}$. For $l\in \mathbb{N}^*$, applying the previous formula to $N+l$ instead of $N+1$, we obtain
$$R_N(z,X)=\displaystyle{-\sum_{k=N+1}^{N+l} \binom{z}{k}X^{(k)}\Delta^{z-k}+R_{N+l}(z,X)},$$
where $R_{N+l}(z,X)$ is convergent for the $Op^{q_X-N-l-1}$ topology. By making $l$ large we can make    $q_X-N-1-l$  as small as we wish, so $R_N(z,X)\in Op^{r{\rm Re}(z)+q_X-N-1}$.
\end{proof}

 We keep  notations of Lemma $\ref{lem:Taylor}$. Let $a$ be a real. On $\mathbb{C}_{<a}$ we have  $R_N(z,X)\in Op^{ q_X+r{\rm Re}(z)-N-1}\subset Op^{ q_X+ra-N-1}$ and $z\mapsto R_N(z,X)$ is holomorphic on $\mathbb{C}_{<N+1}$.
 By making $N$ large we can make  $q_X+ra-N-1$ (resp. $N+1$) as small (resp. large) as we wish. So, if $X\in {\cal D}^{q_X}$ then $([X,\Delta^{z}])_{z \in \mathbb{C}}$ is a holomorphic family of type $q_X-1$. For $N\ge a$, the relation $\eqref{devTaylor}$ is a Taylor expansion $[X,\Delta^{z}]$ of order $q_X+ra-N-1$ and of frontier $a$. From this expansion we obtain that $(T_z)_{z \in \mathbb{C}}$ is a holomorphic family of type $k$ moreover if $X\in {\cal D}^{q_X}$,  then $(XT_z)_{z \in \mathbb{C}}$ et $(T_zX)_{z \in \mathbb{C}}$ are holomorphic families of type $k+q_X$ and $([X,T_z])_{z \in \mathbb{C}}$ is a holomorphic family of type  $k+q_X-1$.

\subsection{Generalized Taylor's expansion}\label{taylor-gen}
In all this paragraph and the next one, we consider an algebra $\left ({\cal D},\Delta,r\right)$ of generalized differential operators and two fixed families of operators:

 $${\cal P}=(P_i)_{1 \le i \le n} \in {\cal D}^n,~~{\cal Q}=(Q_i)_{1 \le i \le n}\in {\cal D}^n.$$
To each couple $(a,b)\in (\mathbb{C}^n)^2$ we associate the operator:

\begin{equation*} \application {H_{a,b}}{{\cal D}}{{\cal D}}{W}{\displaystyle{\sum_{i=1}^n a_i [-Q_i,P_iW]+\sum_{i=1}^n b_i [P_i,Q_iW]}.} \end{equation*}
The next lemma is a  direct application of Taylor's expansion \eqref{lem:Taylor}.

\begin{lem} Let $(a,b)\in (\mathbb{C}^n)^2$. For every $q\in \mathbb{Z}$, we have $H_{a,b}({\rm Hol}(q))\in{\rm Hol}(M+q)$, where $M=\underset{1\le i\le n}\max (order(P_i)+order(Q_i))$.
\end{lem}
A ${\rm Hol}(q)$ family is transformed into a ${\rm Hol}(pM+q)$ family by composing $p$ times such operators. Now let us generalize  Taylor expansions.

\begin{lem}$\label{lem:Taylor Hab X}$
Let $(a,b)\in (\mathbb{C}^n)^2$. Let $X\in {\cal D}^{q_X}$ and $z\in \mathbb{C}$. For every $N\in \mathbb{N}^*$ such that $N+1>{\rm Re}(z)$ we have
\begin{equation*}
H_{a,b}(X\Delta^z)=H_{a,b}(X)\Delta^z+\sum_{k=1}^N \binom{z}{k}A_{a,b}^{(k)}(X)\Delta^{z-k}+R_{N,H_{a,b}}(z,X),
\end{equation*}
with $ A_{a,b}^{(k)}(X)=\sum_{i=1}^n a_i P_iX(Q_i)^{(k)}- b_i Q_iX(P_i)^{(k)}$ for any $k\in  \llbracket1,N\rrbracket$, and
$$R_{N,H_{a,b}}(z,X)=\sum_{i=1}^n a_i P_iXR_N(z,-Q_i)+ b_i Q_iXR_N(z,P_i).$$
Let $M=\underset{1\le i\le n}\max (order(P_i)+order(Q_i))$, then the following assertions hold:
\begin{enumerate}[nolistsep]
\item $H_{a,b}(X)\Delta^{z}\in Op^{r{\rm Re}(z)+q_X+M}$.
\item for every $ k \in \llbracket1,N\rrbracket $, $A_{a,b}^{(k)}(X)\in {\cal D}^{q_X+M+(r-1)k}$.
\item for every $ k \in \llbracket1,N\rrbracket$,  $A_{a,b}^{(k)}(X)\Delta^{z-k}\in Op^{r{\rm Re}(z)+q_X+M-k}$.
\item  $R_{N,H_{a,b}}(z,X)\in Op^{r{\rm Re}(z)+q_X+M-N-1}$.
\end{enumerate}
\end{lem}
\begin{proof}
We have
\begin{equation*}
H_{a,b}(X\Delta^z)=H_{a,b}(X)\Delta^z+\sum_{i=1}^n (a_i P_iX[-Q_i,\Delta^z])+b_iQ_iX[P_i,\Delta^z]).
\end{equation*}

Applying to each commutator the Taylor formula  $\eqref{lem:Taylor}$, we obtain
\begin{equation*}
\begin{split}
H_{a,b}(X\Delta^z)&=H_{a,b}(X)\Delta^z\\[-6pt]
                        &+\sum_{k=1}^N \binom{z}{k} \left (\sum_{i=1}^n a_i P_iX(Q_i)^{(k)}- b_i Q_iX(P_i)^{(k)} \right )\Delta^{z-k}\\[-6pt]
                        &+\sum_{i=1}^n a_i P_iXR_N(z,-Q_i)+ b_i Q_iXR_N(z,P_i).\\
\end{split}
\end{equation*}

\begin{enumerate}
\item We have
       $H_{a,b}(X) \in {\cal D}^{q_X+M}\subset Op^{q_X+M}$ and $\Delta^{z} \in Op^{r{\rm Re}(z)}$,

       then
       $H_{a,b}(X)\Delta^{z}\in Op^{r{\rm Re}(z)+q_X+M}$.

\item Let $k\in  \llbracket1,N\rrbracket$ and let $i \in \llbracket 1,n \rrbracket$,
       $P_iX(Q_i)^{(k)} \in {\cal D}^{q_X+M+(r-1)k}$ and

      $Q_iX(P_i)^{(k)}\in {\cal D}^{q_X+M+(r-1)k}$,
      $A_{a,b}^{(k)}(X)\in {\cal D}^{q_X+M+(r-1)k}$.

\item Let $k\in  \llbracket1,N\rrbracket$, we have
     $A_{a,b}^{(k)}(X)\in {\cal D}^{q_X+M+(r-1)k}\subset{Op}^{q_X+M+(r-1)k}$ and $\Delta^{z-k}\in Op^{r({\rm Re}(z)-k)}$,
     then $A_{a,b}^{(k)}(X)\Delta^{z-k})\in Op^{r({\rm Re}(z)-k)+q_X+M}$.

\item Let $k\in  \llbracket1,N\rrbracket$ and let $i \in \llbracket 1,n \rrbracket$. From $\eqref{lem:Taylor}$, $a_iP_iXR_N(z,-Q_i)+b_iQ_iXR_N(z,P_i)\in Op^{r{\rm Re}(z)+q_X+M-N-1}$  then $R_{N,H_{a,b}}(z)\in Op^{r{\rm Re}(z)+q_X+M-N-1}$.
\end{enumerate}
\end{proof}

\begin{lem}$\label{lem:Taylor HlambdaX}$ Let $ p \in \mathbb{N}^*$ and $\Lambda_p=(a^m,b^m)_{m\in \llbracket1,p\rrbracket}\in (\mathbb{N}^n \times \mathbb{N}^n)^p$. Let
$$H_{\Lambda_p}=H_{a^p,b^p}\circ \dots \circ H_{a^1,b^1}.$$
Let $X\in {\cal D}^{q_X}$ and $z\in \mathbb{C}$. For every $N\in \mathbb{N}^*$ such that $N+1>{\rm Re}(z)$ we have the Taylor expansion:

\begin{equation*}
H_{\Lambda_p}(X\Delta^z)=\sum_{k=0}^N \binom{z}{k}C_{k,\Lambda_p}(X)\Delta^{z-k}+R_{\Lambda_p,N}(z,X),
\end{equation*}
 with $C_{k,\Lambda_p}(X) \in {\cal D}$ and   $R_{\Lambda_p,N}(z,X) \in Op$, for every $ k\in \llbracket0,N\rrbracket$. Moreover,
\begin{enumerate}[nolistsep]
\item for every $ k \in \llbracket1,N\rrbracket$, $C_{k,\Lambda_p}(X)\in {\cal D}^{q_X+pM+(r-1)k}$,
\item for every $ k \in \llbracket1,N\rrbracket$, $C_{k,\Lambda_p}(X)\Delta^{z-k}\in Op^{ r{\rm Re}(z)+q_X+pM-k}$,
\item  $R_{\Lambda_p,N}(z,X)\in Op^{ r{\rm Re}(z)+q_X+pM-N-1}$.
\end{enumerate}
\end{lem}
\begin{proof}
By induction on $p$. The case $p=1$ is the subject of Lemma $\ref{lem:Taylor Hab X}$. Let us assume the result  for the composition of $p$ operators. For  $\Lambda_p=\left ((a_1,b_1),\dots,(a_p,b_p) \right )$ we denote by $H_{\Lambda_p}$ the product $\prod_{m=1}^p H_{a^m,b^m}$. We shall write

\begin{equation*} H_{\Lambda_p}(X\Delta^z)=\sum_{k=0}^N \binom{z}{k}C_{k,\Lambda_p}(X)\Delta^{z-k}+R_{\Lambda_p,N}(z,X). \end{equation*}

 Let $(a,b)\in (\mathbb{N}^n)^2$.  Denote $\Lambda_{p+1}=\left ((a,b),\Lambda_p \right )$ and $H_{\Lambda_{p+1}}=H_{a,b}\circ H_{\Lambda_p}$. From linearity of $H_{a,b}$ we obtain

\begin{equation*} H_{\Lambda_{p+1}}(X\Delta ^z)=\sum_{k=0}^N \binom{z}{k}H_{a,b}(C_{k,\Lambda_p}(X)\Delta^{z-k})+H_{a,b}(R_{\Lambda_p,N}(z,X)). \end{equation*}

For every $k\in \llbracket0,N\rrbracket$, by expanding $H_{a,b}(C_{k,\Lambda_p}(X)\Delta^{z-k})$ to the order $N-k$ according to Lemma $\ref{lem:Taylor Hab X}$, we obtain:

\begin{equation*}
\begin{split}
H_{\Lambda_{p+1}}(X\Delta ^z)&= \sum_{k=0}^N \binom{z}{k}H_{a,b}(C_{k,\Lambda_p}(X))\Delta^{z-k}\\[-6pt]
                                                   &+ \sum_{k=1}^N \left(\sum_{t=0}^{k-1}\binom{z}{t}\binom{z-t}{k-t}A_{a,b}^{(k-t)}(C_{t,\Lambda_p}(X)) \right)\Delta^{z-k}\\[-6pt]
                                                   &+ \sum_{k=0}^N \binom{z}{k}R_{N-k,H_{a,b}}(z-k,C_{k,\Lambda_p}(X))\\[-6pt]
                                                   &+ \sum_{k=0}^N  \binom{z}{k}\left ( \sum_{i=1}^n a_i[P_i,C_{k,\Lambda_p}(X)]R_{N-k}(z-k,-Q_i)+b_i[Q_i,C_{k,\Lambda_p}]R_{N-k}(z-k,P_i) \right )\\[-6pt]
                                                   &+ H_{a,b}(R_{N,\Lambda_p}(z,X)).
\end{split}
\end{equation*}

Denote for every  $k\in\llbracket0,N\rrbracket$,
\begin{equation*}
 C_{k,\Lambda_{p+1}}(X)= H_{a,b}(C_{k,\Lambda_p}(X))+ \left(\sum_{t=0}^{k-1}\binom{k}{t}A_{a,b}^{(k-t)}(C_{t,\Lambda_p}(X)) \right).
\end{equation*}

By grouping the last three terms in $R_{\Lambda_{p+1},N}(z,X)$, we can write

 \begin{equation*}
H_{\Lambda_{p+1}}(X\Delta ^z)=\sum_{k=0}^N \binom{z}{k}C_{k,\Lambda_{p+1}(X)}\Delta^{z-k}
                                                     + R_{\Lambda_{p+1},N}(z,X).
\end{equation*}

Concerning  the orders,
\begin{enumerate}[nolistsep]
\item For $p\in \mathbb{N}^*$, we define the property
$$ P_p:~~~~\forall k \in \mathbb{N},~~C_{k,\Lambda_p}(X)\in {\cal D}^{q_X+pM+(r-1)k}.$$
We shall prove that $P_p$ is true for every $p\ge 1$ by induction on $p$.

For  $p=1$. If $k=0$, $C_{0,\Lambda_1}(X)=H_{a,b}(X)\in {\cal D}^{q_X+M}$. If $k\ge 1$, $C_{k,\Lambda_1}(X)=A_{a,b}^{(k)}(X)\in {\cal D}^{q_X+1.M+(r-1)k}$ from Lemma $\ref{lem:Taylor Hab X}$.\\

Now suppose that $P_p$ is true. If $k=0$, $C_{0,\Lambda_{p+1}}(X)=H_{a,b}\big(C_{0,\Lambda_{p}}(X)\big)\in {\cal D}^{ q_X+(p+1)M+0}$. If $k\ge 1$, for every $t\in \llbracket 0,k-1\rrbracket$ we have $ A_{a,b}^{(k-t)}\in {\cal D}^{M+(r-1)(k-t)}$ from Lemma $\ref{lem:Taylor Hab X}$ and $C_{t,\Lambda_{p+1}}(X)\in {\cal D}^{q_X+pM+(r-1)t}$ by induction hypothesis. Then
$$A_{a,b}^{(k-t)}(C_{t,\Lambda_p}(X))\in {\cal D}^{q_X+M+pM+(r-1)t+(r-1)(k-t)}={\cal D}^{ (p+1)M+(r-1)k}.$$
So
$$\displaystyle{\sum_{t=0}^{k-1}\binom{k}{t}A_{a,b}^{(k-t)}(C_{t,\Lambda_{p+1}}(X))\in {\cal D}^{q_X+ (p+1)M+(r-1)k}}.$$
On the other hand, $\displaystyle{\ H_{a,b}(C_{k,\Lambda_{p}}(X))\in {\cal D}^{  M+q_X+pM+(r-1)k}={\cal D}^{q_X+(p+1)M+(r-1)k}}$. Hence we have proved that $P_{p+1}$ is true.

\item Direct from the previous point.
\item We shall now prove that for every $p\in \mathbb{N}^*$,  $R_{\Lambda_p,N}(z,X))\in Op^{ q_X+r{\rm Re}(z)+pM-N-1}$, by induction on $p$. For $p=1$ it is the Lemma $\ref{lem:Taylor Hab X}$.

For every $p\in \mathbb{N}^*$,
\begin{equation*}
\begin{split}
R_{\Lambda_{p+1},N}(z,X)&= \sum_{k=0}^N \binom{z}{k}R_{N-k,H_{a,b}}(z-k,C_{k,\Lambda_p}(X))\\[-6pt]
&+ \sum_{k=0}^N  \binom{z}{k}\left ( \sum_{i=1}^n a_i[P_i,C_{k,\Lambda_p}(X)]R_{N-k}(z-k,-Q_i)+b_i[Q_i,C_{k,\Lambda_p}(X)]R_{N-k}(z-k,P_i)\right)\\[-6pt]
&+ H_{a,b}(R_{N,\Lambda_p}(z,X)).
\end{split}
\end{equation*}

For every $k\in \llbracket 0,N \rrbracket$ and any $1\le i \le n$, we have $C_{k,\Lambda_p}(X)\in{\cal D}^{q_X+(p+1)M+(r-1)k}$, hence $$R_{N-k,H_{a,b}}(z-k,C_{k,\Lambda_p}(X))\in Op^{q_X+r{\rm Re}(z)+(p+1)M-N-1},$$

 $$[P_i,C_{k,\Lambda_p}(X)]R_{N-k}(z-k,-Q_i))\in Op^{q_X+{r\rm Re}(z)+(p+1)M-N-1},$$

  and $$[Q_i,C_{k,\Lambda_p}(X)]R_{N-k}(z-k,P_i))\in Op^{q_X+{r\rm Re}(z)+(p+1)M-N-1}.$$
 From the induction hypothesis
$R_{N,\Lambda_p}(z)\in Op^{q_X+{r\rm Re}(z)+pM-N-1}$, hence $H_{a,b}(R_{N,\Lambda_p}(z,X))\in Op^{{r\rm Re}(z)+(p+1)M-N-1}$.
So $R_{\Lambda_{p+1},N}(z)\in Op^{{r\rm Re}(z)+(p+1)M-N-1}$, and  Lemma $\ref{lem:Taylor HlambdaX}$ is proved.
\end{enumerate}
\end{proof}

To  $\Lambda_p=(a^m,b^m)_{m\in \llbracket1,p\rrbracket}\in (\mathbb{R}^n \times \mathbb{R}^n)^p$ and $\Gamma_p=(\alpha_m,\beta_m)_{m\in \llbracket1,p\rrbracket} \in (\mathbb{R}^2)^p$ we associate
 two sequences  $(T_s)_{s \in \mathbb{N}^*}$ and a family $(T_{z'})_{z' \in \mathbb{C}}$ of operators in  ${\rm End}(\Hi^{\infty})$ defined by:

\begin{equation*}\forall s\in \mathbb{N}, ~~T_s=\prod_{m=1}^p \left (H_{a^m,b^m}-(\alpha_m s+ \beta_m)Id_{Op}  \right ).\end{equation*}

\begin{equation*}\forall z'\in \mathbb{C},~~ T_{z'}=\prod_{m=1}^p \left (H_{a^m,b^m}-(\alpha_m z'+ \beta_m)Id_{Op}  \right ).\end{equation*}

We denote by $b$ the polynomial function:

 \begin{equation*} \forall z'\in \mathbb{C},~~b(z')=(-1)^{p+1}\prod_{m=1}^p (\alpha_m z'+ \beta_m), \end{equation*}

and by $H_{z'}$ the operator:

\begin{equation*}\forall z'\in \mathbb{C},~~ H_{z'}=b(z')Id_{Op}+ T_{z'}. \end{equation*}

We can remark that $H_{z'}$ is a composition of commutators.

\begin{lem}\label{lem:Taylor Hzx} Let $X\in {\cal D}_q$. Let $(z,z') \in \mathbb{C}^2$. For every $N\in \mathbb{N}^*$ such that $N+1>{\rm Re}(z)$ we have:
\begin{equation*} ~H_{z'}(X\Delta^z)=\sum_{k=0}^N \binom{z}{k}D_k(z',X)\Delta^{z-k}+R_{N}(z,z',X),
\end{equation*}
for every  $k\in \llbracket0,N\rrbracket$, $D_k(z',X) \in {\cal D}_{p}[z']$ (polynomial in $z'$  with degree less than or equal to  $p$  and with coefficients in ${\cal D}$). The orders behave as follows:
\begin{enumerate}[nolistsep]
\item For every $k\in \llbracket0,N\rrbracket$, $D_k(z',X)\in {\cal D}^{ q_X+pM+(r-1)k}$.
\item For every $k\in \llbracket0,N\rrbracket$, $D_k(z',X)\Delta^{z-k}\in Op^{ r{\rm Re}(z)+q_X+pM-k}$.
\item We have $R_{N}(z,z',X)\in Op^{r{\rm Re}(z)+q_X+pM-N-1}.$
\end{enumerate}
\end{lem}

\begin{proof} Let $I$ be the set of subsets of $\llbracket 1,p \rrbracket$. For $u\in I$, denote by $\overline{u}$ the complementary subset $\llbracket 1,p\rrbracket \setminus u$. For $u\in I$, let $|u|$ be the length of $u$.   If $u \ne \emptyset$ we  write
 $H_{\Lambda_{u}}=\prod_{i=1}^{|u|} H_{a^{u_i},b^{u_i}}$  where the elements $u_i$ of $u$ are displayed in increasing order, and $b^u(z')=(-1)^{|u|}\prod_{i=1}^{|u|} (\alpha_{u_i}z'+\beta_{u_i})$.  By convention $H_{\Lambda_{\emptyset}}=Id_{{\cal D}}$, and $b^{\emptyset}(z)=1$. With these notations we can express $T_{z'}$ and $H_{z'}$ as follows:

 \begin{equation*} T_{z'}=\sum_{u\in I} b^{\overline{u}}(z')H^u~~~{\text and}~~~ H_{z'}=\sum_{u\in I,|u|\ge1} b^{\overline{u}}(z')H^u.\end{equation*}
For $u \in I, |u| \ge 1$,  $H_{\Lambda_{u}}$ is a $H_{\Lambda_p}$ type operator, with $p=|u|$.
  Lemma $\ref{lem:Taylor HlambdaX}$ implies

 \begin{equation*} H_{\Lambda_{u}}(X\Delta^z)=\sum_{k=0}^N \binom{z}{k}C_{k,\Lambda_{u}}(X)\Delta^{z-k}+R_{\Lambda_{u},N}(z,X). \end{equation*}

 \noindent We deduce

  \begin{equation*} H_{z'}(X\Delta^z)=\sum_{u\in I,|u|\ge1} b^{\overline{u}}(z')\left (\sum_{k=0}^N \binom{z}{k}C_{k,\Lambda_{u}}(X)\Delta^{z-k}+R_{\Lambda_{u},N}(z,X)\right), \end{equation*}

\noindent and by permuting the sums

 \begin{equation*} H_{z'}(X\Delta^z)=\sum_{k=0}^N\left (\sum_{u\in I,|u|\ge1} b^{\overline{u}}(z')C_{k,\Lambda_{u}}(X)\right)\Delta^{z-k}+\sum_{u\in I,|u|\ge1} b^{\overline{u}}(z')R_{\Lambda_{u},N}(z,X). \end{equation*}

\noindent Setting
 \begin{equation*} \forall  k\in \llbracket0,N\rrbracket,~~D_k(z',X)=\sum_{u\in I,|u|\ge1} b^{\overline{u}}(z')C_{k,\Lambda_{u}} \end{equation*}
 and
 \begin{equation*} R_{N}(z,z',X)=\sum_{u\in I,|u|\ge1} b^{\overline{u}}(z')R_{\Lambda_{u},N}(z,X). \end{equation*}

\noindent We obtain We obtain

 \begin{equation*} ~H_{z'}(\Delta^z)=\sum_{k=0}^N \binom{z}{k}D_k(z',X)\Delta^{z-k}+R_{N}(z,z',X). \end{equation*}

\noindent The computation of orders  is  a direct consequence of Lemma $\ref{lem:Taylor HlambdaX}$.
\end{proof}
\subsection{Meromorphic continuation}
$\label{suitered}$

\begin{df}$\label{suitereduc}$
Let $\left ({\cal D},\Delta,r\right)$ be an algebra of generalized differential operators. We keep the notations of Paragraph \ref{taylor-gen}. The sequence $(T_s)_{s \in \mathbb{N}}$ is called a  reduction sequence if there exist an integer $p\in \mathbb{N}^*$, a family  $\Lambda_p=(a^m,b^m)_{m\in \llbracket1,p\rrbracket}\in (\mathbb{R}^n \times \mathbb{R}^n)^p$ and a nonzero family $\Gamma_p=(\alpha_m,\beta_m)_{m\in \llbracket1,p\rrbracket} \in (\mathbb{R}^2)^p$ such that:

\begin{enumerate}[nolistsep]
\item For every $\displaystyle{s\in \mathbb{N}^*,~~ T_s=\prod_{m=1}^p \left(H_{a^m,b^m}-(\alpha_m s+ \beta_m)Id_{Op}  \right)}$.

\item For every $X \in {\cal D}^{q_X}$  and  every $i \in \mathbb{N}$,
$\displaystyle{ T_{q_X+ri}(X\Delta^i) \in {\cal D}[\Delta]^{q_X+ri-1}}$.

\end{enumerate}

\end{df}

To a reduction sequence $(T_s)_{s \in \mathbb{N}}$ as above, we associate a polynomial function $b$ and two families $(T_{z'})_{z' \in \mathbb{C}}$ and $(H_{z'})_{z' \in \mathbb{C}}$ of operators defined by:

 \begin{equation*}
 \begin{split}
 \forall z' \in \mathbb{C},~ &b(z')=(-1)^{p+1}\prod_{m=1}^p (\alpha_m z'+ \beta_m),\\[-2pt]
\label{defHz} \forall z' \in \mathbb{C},~&T_{z'}=\prod_{m=1}^p  \left (H_{a^m,b^m}-(\alpha_m z'+ \beta_m)Id_{Op} \right)\\[-2pt]
 \forall z' \in \mathbb{C},~&H_{z'}=T_{z'}+b(z')Id_{Op}
 \end{split}
 \end{equation*}

 The next theorem is a crucial point. We shall extend the reduction from a sequence to a family indexed by $\mathbb{C}$.

\begin{thm}\label{lem:Reductionmero} Let $\left ({\cal D},\Delta,r\right)$ be an algebra of generalized differential operators. Let $(T_s)_{s \in \mathbb{N}}$ be a  reduction sequence, and $(H_{z'})_{z' \in \mathbb{C}}$ defined as in the paragraph $\eqref{defHz}$. For every  $X \in {\cal D}^{q_X}$ we have
\begin{equation*} \forall z \in \mathbb{C}, ~~H_{rz+q_X}(X\Delta^z)=b(rz+q_X)X\Delta^z+ R_z,
\end{equation*}
with $(R_z)_{z\in \mathbb{C}}\in{\rm Hol}(q-1)$.
\end{thm}
\begin{proof} Let $X\in {\cal D}^{q_X}$, $z \in \mathbb{C}$ and $N\in \mathbb{N}^*$ such that $N+1>{\rm Re}(z)$. Applying Lemma $\ref{lem:Taylor Hzx}$, we get:
\begin{equation*} ~H_{rz+q_X}(X\Delta^z)=\sum_{k=0}^N \binom{z}{k}D_k(X,rz+q_X)\Delta^{z-k}+R_{N}(X,z).
\end{equation*}
For every $k \in  \llbracket 0,N \rrbracket$,  $ \binom{z}{k}D_k(X,rz+q_X)$ is a polynomial in $z$ of degree less than or equal to   $k+p-1$, hence  less than or equal to   $N+p-1$  with coefficients in  ${\cal D}$. Denote by $(L_i)_{i \in \llbracket 0,N+p-1 \rrbracket}$ the family of   Lagrange interpolating polynomials, that is, the polynomials defined by:
\begin{equation*} \forall(i,j) \in \llbracket 0,N+p-1 \rrbracket^2, L_i(j)=\delta_i^j.
\end{equation*}
\noindent For every $k \in  \llbracket 0,N \rrbracket$ we have:
\begin{equation*} \forall z \in \mathbb{C}, \binom{z}{k}D_k(rz+q_X)=\sum_{i=0}^{N+p-1} \binom{i}{k}D_k(X,ri+q)L_i(z), \end{equation*}

\noindent and therefore

\begin{equation*} H_{rz+q_X}(X\Delta^z)=\sum_{k=0}^N \left ( \sum_{i=0}^{N+p-1} \binom{i}{k}D_k(X,ri+q_X)L_i(z) \right )\Delta^{z-k}+R_{N}(X,z). \end{equation*}

\noindent Permuting the sums we obtain

\begin{equation*} H_{rz+q_X}(X\Delta^z)=  \sum_{i=0}^{N+p-1}\left (\sum_{k=0}^N \binom{i}{k}D_k(X,ri+q_X)\Delta^{i-k}\right )L_i(z) \Delta^{z-i}+R_{N}(X,z), \end{equation*}

\noindent and so

\begin{equation*}
\begin{split}
H_{rz+q_X}(X\Delta^z)&=\sum_{i=0}^{N+p-1}\left (\sum_{k=0}^{N+p-1} \binom{i}{k}D_k(X,ri+q_X)\Delta^{i-k}\right )L_i(z) \Delta^{z-i} \\[-6pt]
                     &-\sum_{i=0}^{N+p-1}\left (\sum_{k=N+1}^{N+p-1} \binom{i}{k}D_k(X,ri+q_X)\Delta^{i-k}\right )L_i(z) \Delta^{z-i}     \\[-6pt]
                     &+R_{N}(X,z).                             \\
\end{split}
\end{equation*}

\noindent For every $i\in \llbracket 0,N+p-1 \rrbracket$ the following  formulas hold:

\begin{equation*} ~H_{ri+q_X}(X\Delta^i)=\sum_{k=0}^{i} \binom{i}{k}D_k(X,ri+q_X)\Delta^{i-k}=\sum_{k=0}^{N+p-1} \binom{i}{k}D_k(X,ri+q_X)\Delta^{i-k}. \end{equation*}

\noindent So we have

\begin{equation*}
\begin{split}
H_{rz+q_X}(X\Delta^z)&=\sum_{i=0}^{N+p-1}H_{ri+q_X}(X\Delta^i)L_i(z) \Delta^{z-i} \\[-6pt]
                     &-\sum_{i=0}^{N+p-1}\left (\sum_{k=N+1}^{N+p-1} \binom{i}{k}D_k(X,ri+q_X)\Delta^{i-k}\right )L_i(z) \Delta^{z-i}                   \\[-6pt]
                     &+R_{N}(X,z).                             \\
                     \end{split}
\end{equation*}

\noindent By definition of the operators $H$ and $T$, we get for every $i\in \llbracket 0,N+p-1 \rrbracket$

\begin{equation*}  H_{ri+q_X}(X\Delta^i)=b(ri+q_X)X\Delta^i+T_{ri+q_X}(X\Delta^i), \end{equation*}

\noindent hence

\begin{equation*}
\begin{split}
H_{rz+q_X}(X\Delta^z)&=\left (\sum_{i=0}^{N+p-1}b_{ri+q_X}L_i(z)\right ) X\Delta^{z} \\[-6pt]
                     &+\sum_{i=0}^{N+p-1}L_i(z)T_{ri+q_X}(X\Delta^i) \Delta^{z-i} \\[-6pt]
                     &-\sum_{i=0}^{N+p-1}\sum_{k=N+1}^{N+p-1} \binom{i}{k}L_i(z) D_k(X,ri+q_X)\Delta^{z-k}\\[-6pt]
                     &+R_{N}(X,z).                             \\
\end{split}
\end{equation*}

By definition of a  reduction sequence, we get
 $T_{ri+q_X}(X\Delta^i)\in {\cal D}[\Delta]^{q_X+ri-1} $ for every $i\in \llbracket 0,N+p-1\rrbracket$. So we obtain  $$\left(\sum_{i=0}^{N+p-1}L_i(z)T_{2i+q_X}(X\Delta^i) \Delta^{z-i}\}\right)_{z\in \mathbb{C}}\in{\rm Hol}(q-1).$$

Let $N$ such that $N \ge Mp$, for every $i\in \llbracket 0,N+p-1 \rrbracket$ and every $k\in \llbracket N+1,N+p-1\rrbracket$, $\left(D_k(X,ri+q)\Delta^{z-k}\right)_{z\in \mathbb{C}}\in{\rm Hol}(q-1)$ from $\ref{lem:Taylor Hzx}$. As $b$ is a polynomial with degree $p$ such that $p \le N+p-1$, using Lagrange interpolation we get the formula   $\displaystyle{\sum_{i=0}^{N+p-1}b_{ri+q}L_i(z)=b(rz+q)}$.
 We have proved $ H_{rz+q}(X\Delta^z)=b(rz+q)X\Delta^z+ R_z$, where $(R_z)_{z\in \mathbb{C}}\in{\rm Hol}(q-1)$.
\end{proof}

\begin{cor}\label{cor:ext} Let $\left ({\cal D},\Delta,r\right)$ be an algebra of generalized differential operators. Let $q \in \mathbb{Z}$ and $X\in {\cal D}$.
\begin{enumerate}[nolistsep]
\item Let $(X\Delta^{z-l})_{z\in \mathbb{C}}$  be an elementary holomorphic family of type $q$ then
$$H_{rz+q}(X\Delta^{z-l})=b(rz+q)X\Delta^{z-l}+ S_z $$
where $(S_z)_{z\in \mathbb{C}}\in{\rm Hol}(q-1)$.
\item Let $(T_z)_{z\in \mathbb{C}}$ be a holomorphic family of type $q$ then $H_{rz+q}(T_z)=b(rz+q)T_z+ S_z $
where $(S_z)_{z\in \mathbb{C}}\in{\rm Hol}(q-1)$.
\end{enumerate}
\end{cor}
\begin{proof}
\begin{enumerate}
\item We have $X\in {\cal D}^{ rl+k}$, so:
$$H_{rz+q}(X\Delta^{z-l})=H_{r(z-l)+rl+k}(X\Delta^{z-l})=b(r(z-l)+rl+q)X\Delta^{z-l}+R_z=b(rz+q)X\Delta^{z-l}+ S_z,$$
with $(S_z)_{z\in \mathbb{C}}\in{\rm Hol}(q-1)$.
\item   Let $(T_z)_{z\in \mathbb{C}}$ be a holomorphic family of type $q$. Let $a \in \mathbb{R}$ and $m\in \mathbb{R}$, we can write $T_z=T_z^1+\dots+T_z^q+R_z$,  where $T_z^j$ is an elementary holomorphic family of type $q$ for every $1\le j \le q$ and the remainder  $z\mapsto R_z$  is holomorphic from the half-plane ${\rm Re}(z)<a$ to $Op^{m-Mp}$. For every $z \in \mathbb{C}$, by linearity $H_{rz+q}(T_z)=H_{rz+q}(T_z^1)+\dots+H_{rz+q}(T_z^q)+H_{rz+q}(R_z)$.
    For each $1 \le j \le q, H_{rz+q}(T_z)=b(rz+q)T_z^j+ S_z^j$
where $(S_z^j)_{z\in \mathbb{C}}\in{\rm Hol}(q-1)$ by the previous point. For every $z\in \mathbb{C}$ and every $t\in \mathbb{R}$, the operator $H_{rz+q}$ has polynomial coefficients in $z$ and maps $Op^{t}$ to $Op^{t+Mp}$, thus the remainder $z \mapsto H_{rz+q}(R_z)$ is holomorphic from the half-plane ${\rm Re}(z)<a$ to $Op^{m}$. Hence we have proved that $H_{rz+q}(T_z)=b(rz+q)T_z+ S_z $, where $(S_z)_{z\in \mathbb{C}}\in{\rm Hol}(q-1)$.
\end{enumerate}
\end{proof}

\begin{thm}\label{thetheorem} Let $\left ({\cal D},\Delta,r\right)$ be an algebra of  generalized differential operators, and suppose
that there exists a  reduction sequence. Denote by $b$ its associated   polynomial function and by ${\rm Rac(b)}$ its set of complex roots. We suppose that $\Delta^{-\frac{1}{r}} \in {\cal L}^p(\Hi)$ for some real $p \ge 1$. Let  $X \in {\cal D}^{q_X}$. The zeta spectral function
\begin{equation*} \application{\zeta_{X,\Delta}}{\mathbb{C}_{<-\frac{p+q_X}{r}}}{\mathbb{C}}{z}{{\rm Trace}(X\Delta^{z})} \end{equation*}
is holomorphic and admits a meromorphic continuation, also denoted by $\zeta_{X,\Delta}$, to all of $\mathbb{C}$ . The poles of $\zeta_{X,\Delta}$ are constrained in a finite union of arithmetic progressions:
\begin{equation*} {\rm Poles}(\zeta_{X,\Delta})\subset \underset{\alpha \in {\rm Rac}(b)}\bigcup R_{\alpha}, \end{equation*}
with $R_{\alpha}=\left \{\frac{\alpha-q_X}{r}+\frac{l}{r}|l \in \mathbb{Z}~{\text and}~l\ge-p-\alpha\right \}$. Let $\omega \in {\rm Poles}(\zeta_{X,\Delta})$, denoting $m(\omega)$ his pole multiplicity order, we have:
\begin{equation*} m(\omega) \le \sum_{\alpha \in {\rm Root(b)}}m(\alpha)\mathbf{1}_{R_{\alpha}}(\omega),
\end{equation*}
with $m(\alpha)$ denoting the multiplicity order of the root $\alpha$.
\end{thm}

\begin{proof} Let  $X \in {\cal D}^{q_X}$, $(T_z)$ a holomorphic family of type $q_X$ and $p$   a real greater than 1 such that  $\Delta^{-\frac{1}{r}} \in {\cal L}^p(\Hi)$. From Proposition $\ref{holozeta2}$, the application $z \mapsto {\rm Trace}(T_z)$ is holomorphic on the open half-plane $\mathbb{C}_{-\frac{p+q_X}{r}}$. From  Corollary $\ref{cor:ext}$ ,  we have $H_{rz+q_X}(T_z)=b(rz+q_X)T_z+S_z$, where $(S_z)\in{\rm Hol}(q_X-1)$. For ${\rm Re}(z)<<0$, the operators $H_{rz+q_X}(T_z)$,$T_z$ and $S_z$ are trace class and moreover ${\rm Trace}(H_{rz+q_X}(T_z))=0$ from the Proposition $\ref{comutnulle}$. For ${\rm Re}(z)<<0$,   we obtain $b(rz+q_X){\rm Trace}(T_z)={\rm Trace} (-S_z)$. For $l \in \mathbb{N}$, repeating $l+1$ times the above operation we can write
 $$b(rz+q_X-l)\dots b(rz+q_X-1)b(rz+q_X){\rm Trace}(T_z)={\rm Trace} (U_z),$$ with $(U_z)\in{\rm Hol}(q_X-l-1)$ . The function $z \mapsto {\rm Trace}(U_z)$ is holomorphic on the open half-plane $\mathbb{C}_{-\frac{p+q}{r}+\frac{l+1}{r}}$  as consequence of Proposition $\ref{holozeta2}$, we conclude that $z\mapsto {\rm Trace}(T_z)$ extends to a meromorphic function on  $\mathbb{C}$. The localization and upper bound of multiplicity for poles are immediate consequences of the expression $\displaystyle{{\rm Trace}(T_z)=\frac{1}{b(rz+q_X-l)\dots b(rz+q_X-1)b(rz+q_X)}{\rm Trace} (U_z)}$.
\end{proof}

\section{Application to a family of nilpotent Lie algebras}\label{sect:deux}
\subsection{The sets $O(I)$ and $T(I)$}\label{sec1}

The notations and results of this section come in a large part from the article \cite{grobner5}.
Let $\lig$ be a Lie algebra with basis $(X_1,\dots,X_n)$. Denote by ${\cal S}$ the set of monomial ${\cal S}=\{X_1^{\alpha_1}\dots X_n^{\alpha_n}~|~(\alpha_1,\dots,\alpha_n)\in \mathbb{N}^n \}$.  Theorem of Poincaré-Birkhoff-Witt  states that ${\cal S}$ is a basis of the complex vector space ${\cal U}(\lig)$. The degree of $U= X_1^{\alpha_1}\dots X_n^{\alpha_n}\in {\cal S} $  is defined by
$\deg(U)=\alpha_1+\dots+\alpha_n$. We denote <  the total degree ordering on ${\cal S}$, which is defined by:$\label{degtot}$
$$U_1:= X_1^{\alpha_1}\dots X_n^{\alpha_n} <X_1^{\beta_1}\dots X_n^{\beta_n}=:U_2$$
if and only if $\left [\deg(U_1)<\deg(U_2) \right ]$ or $\left [\deg(U_1)=\deg(U_2)~\hbox{ and}~~\exists j\in \llbracket 1,n \rrbracket~~\hbox{ such~ that} ~\alpha_i=\beta_i~~ \hbox{ if}~~i<j,~\hbox{ and}~~ \alpha_j<\beta_j \right ]$.

 The total degree relation on ${\cal S}$ is a good ordering. Let $U \in{\cal U}(\lig)$ be  a non-zero element. There exists an unique decomposition as follows:
$$U=\sum_{i=1}^s c_iU_i,~~s\in \mathbb{N}^*,~c_i \in \mathbb{C}^*, U_i \in S~{\text and}~U_1>U_2>\dots>U_s.$$

So we define $T(U)=U_1$  the maximal term  and $cd(U)=c_1$ the leading coefficient of $U$. Let $I$ be an ideal of the universal enveloping algebra ${\cal U}(\lig)$. We denote
 \begin{equation*}
 \begin{split}
 T(I)&=\{T(U)~|~U\in I \}, \\[-2pt]
 O(I)&={\cal S}\setminus T(I).
\end{split}
\end{equation*}
We also consider on ${\cal S}$ a partial ordering $\preceq$, defined by

 $$X_1^{\alpha_1}\dots X_n^{\alpha_n} \preceq X_1^{\beta_1}\dots X_n^{\beta_n}~~\hbox{ if and only  if}~~ \forall i \in \llbracket 1,n \rrbracket,~\alpha_i \le \beta_i.$$

\begin{lem}$\label{lemsecfini}$ Let $I$ be an ideal of  ${\cal U}(\lig)$.
 The set $T(I)$ is an upper set for the ordering $\preceq$. Namely, for every $V \in T(I)$ and $W \in {\cal S}$ if $V \preceq W$ then $W \in T(I)$.
\end{lem}

\begin{proof} Let $V= X_1^{\alpha_1}\dots X_n^{\alpha_n}\in T(I)$ and $W=  X_1^{\beta_1}\dots X_n^{\beta_n}\in {\cal S} $. There exists $U \in I$ such that $T(U)=V$. If $V \preceq W$, then $X_1^{\beta_1-\alpha_1}\dots X_n^{\beta_n-\alpha_n} \in {\cal U}(\lig)$. As $I$ is an  ideal $X_1^{\beta_1-\alpha_1}\dots X_n^{\beta_n-\alpha_n}U \in I$, and $T(X_1^{\beta_1-\alpha_1}\dots X_n^{\beta_n-\alpha_n}U)=W$ then $W \in T(I)$.
\end{proof}

\noindent The following proposition corresponds to \cite[Theorem 1.3]{grobner5}.
\begin{pr} $\label{pr:OI}$ Let $I$ be an ideal of  ${\cal U}(\lig)$. We have the decomposition
 $${\cal U}(\lig)=I \oplus Vect(O(I))$$
So, for every $U\in {\cal U}(\lig)$ there exists one and only one $V=Can(U,I)\in Vect(O(I))$ such that $U-V \in I$.
\end{pr}

\begin{proof} The two last items are clearly equivalent to the first one, so we prove $(1)$. If  $U\ne 0$ is an element of $I \cap Vect(O(I))$ then $T(U)$ belongs to $T(I) \cap O(I)$, this is a contradiction so $I \cap Vect(O(I))=\{0\}$. We shall prove that ${\cal U}(\lig)=I + Vect(O(I))$ with an algorithm.

\medskip
{\tt
$U_0:=U$, $\Phi_0:=0$, $H_0:=O$, $i:=0$.

~~While $U_i \ne 0$ do:

~~~~~If $T(U_i) \notin T(I)$ then

~~~~~~~~$\Phi_{i+1}:=\Phi_i, H_{i+1}:=H_i+lc(U_i)T(U_i), U_{i+1}:=U_i-lc(U_i)T(U_i)$

~~~~~Else

~~~~~~~~Choosing $G\in I$ such that $T(G)=T(U_i)$ and $lc(G)=1$

~~~~~~~~$\Phi_{i+1}:=\Phi_i+lc(U_i)G, H_{i+1}:=H_i, U_{i+1}:=U_i-lc(U_i)G$

~~$i:=i+1$

$\Phi:=\Phi_i, H=H_i$.
}
\medskip

 For every $i$, $\Phi_i \in I,~~H_i \in Vect(O(I))$ and $U_i+\Phi_i+H_i=f$ is a loop invariant. From a loop to the next $T(U_{i+1})<T(U_i)$,  as the ordering $\preceq $ is good  the algorithm finishes.
\end{proof}

\begin{pr} $\label{pr:OImin}$ Let $I$ be an ideal of ${\cal U}(\lig)$. If $U \in {\cal U}(\lig)$ and $\Phi \in I$ then $T(\Phi+Can(U,I))\ge T(Can(U,I))$.
\end{pr}
\begin{proof} Let $U \in {\cal U}(\lig)$. If there exists $\Phi \in I$ such that $T(\Phi+Can(U,I)) < T(Can(U,I))$ then $cd(\Phi)=-cd(Can(U,I))$ and
$T(\Phi)=T(Can(U,I))\ne 0 \in T(I) \cap O(I)$, this is a contradiction.
\end{proof}

\subsection{ The Lie algebras $\lig_{\alpha,{\cal I}}$}
$\label{liealpha}$

  Let $\gamma=(\gamma_i)_{1\le i\le n}\in \mathbb{N}^n$ and $\beta=(\beta_i)_{1\le i\le n}\in \mathbb{N}^n$. We define the order $\gamma \preceq \beta$ by  $ \gamma_i \le \beta_i$ for every $ i \in \llbracket 1,n \rrbracket$. Denote  $|\gamma|=\sum_{i=1}^n \gamma_i$,  $\gamma!=\prod_{i=1}^n \gamma_i!$ and if  $\gamma \preceq \beta$, $\displaystyle{\binom{\beta}{\gamma}=\frac{\beta!}{\gamma!(\beta-\gamma)!}}$. Let $i\in \llbracket 1,n \rrbracket$, denote $\delta^i=(\delta_k^i)_{1\le k\le n}$, where $\delta_k^i$ is the Kronecker symbol, the $n$-uple with all zero components  excepted the $i$-th equal to 1. We now define the class of Lie algebras $\lig_{\alpha,{\cal I}}$. Let $n\in \mathbb{N}^*$ and $\alpha=(\alpha_i)_{1\le i\le n}\in (\mathbb{N}^*)^n$. Let $p\in \mathbb{N}^*$ and ${\cal I}=\{I_j |1 \le j\le p\}$ a  partition of $\llbracket 1,n \rrbracket$ in $p$ subsets. By convention, the subsets $I_j$ are sorted according to their smallest element, that is $j <j'$ if and only if $\min(I_j)<\min(I_{j'})$. For $j \in \llbracket 1,p \rrbracket$, we denote $\displaystyle{A^j=\{ \beta \in \mathbb{N}^n| \beta \preceq \sum_{i\in I_j}\alpha_i\delta^i \}}$ and   $A^{\cal I}=\displaystyle{\bigcup_{j=1}^p A^j}$. The Lie algebra $\lig_{\alpha,{\cal I}}$ is defined by:

 \begin{equation*}\lig_{\alpha,{\cal I}}=\langle X_1,\dots,X_n,Y^{\beta}|~\beta \in A^{\cal I} \rangle_{\mathbb{C}-ev}, \end{equation*}

\noindent and the only non trivial relations are

 \begin{equation*}  \forall i \in \llbracket 1,n \rrbracket,~ \forall \beta \in A^{\cal I}, ~~ [X_i,Y^{\beta}]=Y^{\beta-\delta^i}~\text{if}~\beta_i\ge 1. \end{equation*}

\noindent We can easily prove that $\lig_{\alpha,{\cal I}}$ is a nilpotent Lie algebra. The family
$${\cal S}= \left\{X_1^{p_1}\dots X_n^{p_n} \prod_{\beta \in A^{\cal I}}(Y^{\beta})^{q_{\beta}}~|~p_i \in \mathbb{N}, ~q_{\beta}\in \mathbb{N} \right\}$$
is a  basis of $U(\lig_{\alpha,{\cal I}})$. This basis is ordered with the total degree ordering  $\eqref{degtot}$, here the lexical order is $X_1 <\dots<X_n<Y^{\beta}$ for every $\beta$, and the elements $Y^{\beta}$ are sorted with the natural lexical ordering of $n$-uples $\beta$ in $\mathbb{N}^n$.

\subsection{ Generators of the ideal $I(f)$}
$\label{sec3}$
Let  $\lig_{\alpha,{\cal I}}=\langle X_1,\dots,X_n,Y^{\beta}|~\beta \in A^{\cal I} \rangle$ be one of the nilpotent Lie algebras defined in Paragraph $\eqref{liealpha}$. We denote by $\rho$ the infinitesimal representation of $\lig_{\alpha,{\cal I}}$ associated to the linear functional $f=(Y^{(0,\dots,0)})^*$ and the polarisation $ \langle Y^{\beta}|~\beta \in A^{\cal I} \rangle  $ via the Kirillov correspondence. Let ${\cal P}(\mathbb{R}^n)$ be the  algebra of linear differential operators with polynomial coefficients on $\mathbb{R}^n$ acting on ${\cal C}_c^{\infty}(\mathbb{R}^n)$. For every $\gamma=(\gamma_i)_{1\le i\le n}\in \mathbb{N}^n$, set  $x^{\gamma}:=x_1^{\gamma_1}\dots x_n^{\gamma_n}$. We can give an explicit expression of $\rho$, as an algebra homomorphism from $\lig_{\alpha,{\cal I}}$ to ${\cal P}(\mathbb{R}^n)$ (see \cite[section 1]{Multifili1}):

\begin{equation} \forall k \in \llbracket 1,n \rrbracket, \rho(X_k)=-\frac{\partial}{\partial x_k}, ~~~~~~~~\forall \beta \in A^{\cal I}, \rho(Y^\beta)=i\frac{(-1)^{|\beta|}x^{\beta}}{\beta!}. \end{equation}

Here $\rho(Y^\beta)$ is the multiplication operator by $i\frac{(-1)^{|\beta|}x^{\beta}}{\beta!}$. The isotropic algebra of $\lig_{\alpha,{\cal I}}$ associated to $f$ is defined by
$$\lig_{\alpha,{\cal I}}^f=\left\{ X \in \lig_{\alpha,{\cal I}}| \forall Y \in \lig_{\alpha,{\cal I}}, f([X,Y])=0 \right\}.$$

\begin{lem}$\label{genIf}$ Let be a Lie algebra $\lig_{\alpha,{\cal I}}$, then
$$\lig_{\alpha,{\cal I}}^f= \left\langle Y^{(0,\dots,0)},Y^\beta ~|~\beta \in A^{\cal I},|\beta| \ge2 \right\rangle_{\mathbb{C}-ev}.$$ \end{lem}

\begin{proof} Let $X=\sum_{j=1}^n a_j X_j +\sum_{\beta \in  \alpha^{\cal I}}b_{\beta}Y^{(\beta)}\in \lig_{\alpha,{\cal I}}^f$. For every $j\in \llbracket 1,n \rrbracket$,  $f([X,Y^{\delta^j}])=0$ (resp. $f([X,X_j))=0$) then $a_j=0$ (resp. $b_{\delta^j}=0$) hence $\lig_{\alpha,{\cal I}}^f \subset \langle Y^{(0,\dots,0)},Y^\beta~|~\beta \in A^{\cal I},|\beta| \ge2\rangle_{ev}$. The other inclusion is easy.
\end{proof}

The next lemma is an application of the article \cite[page 304]{Godfrey}, and will give us a basis of the ideal $I(f)$.

\begin{lem}$\label{lemgenIf}$ Let be a Lie algebra $\lig_{\alpha,{\cal I}}$.
Denote by $\displaystyle{\Gamma_j=\sum_{k\ge 0}\frac{(i)^k}{k!} (Y^{\delta_j})^k (ad X_j)^k}$ for every $j\in \llbracket 1,n \rrbracket$ and by $\displaystyle{\Gamma=\prod_{j=1}^n \Gamma_j}$. Then $I(f)$ is the ideal of $U(\lig_{\alpha,{\cal I}})$ generated by

$$ \left \{Y^{(0,\dots,0)}-i,~~ \Gamma(iY^\beta)~|~\beta \in A^{\cal I},|\beta| \ge2 \right \}$$
\end{lem}

 The operators $\Gamma_j$ commute with each other, so the product for $\Gamma$ is defined with no more precaution on the order of factors. We shall give in Corollary $\ref{corgenIf}$ another basis of $I(f)$ which is more interesting for studying  $T(I(f))$.

For every $\gamma \in \alpha^{\cal I}$, we denote $Y_*^{\gamma}=(i)^{1-|\gamma|}\displaystyle{\prod_{k=1}^n(Y^{{\delta^k})^{\gamma_k}}}$.  For example, if $\gamma=(2,3)\in \alpha^{\cal I}$,
$Y^{\gamma}=Y^{(2,3)} \in \lig_{\alpha,{\cal I}}$, whereas $Y_*^{\gamma}=(Y^{(1,0)})^2Y^{(0,1)})^3U \in U(\lig_{\alpha,{\cal I}})$.

\begin{lem}$\label{lemgamma1}$ Let be a Lie algebra $\lig_{\alpha,{\cal I}}$. The notations are these of Lemma \ref{lemgenIf}.

\begin{enumerate}[nolistsep]
\item For every $\beta \in A^{\cal I}$ and every $j\in \llbracket 1,n \rrbracket$,

$\label{eq1gamma}  \Gamma_j(Y^{\beta})=\displaystyle{\sum_{k=0}^{\beta_j} \frac{i^k}{k!}(Y^{\delta^i})^kY^{\beta-k \delta^j}}~~~é$ and
  $~~~\Gamma(iY^{\beta})=\displaystyle{\sum_{\gamma \le \beta} \frac{(-1)^{|\gamma|}}{\gamma !}Y_*^{\gamma}Y^{\beta-\gamma}}.$
\item For every $\beta \in A^{\cal I}\setminus \{(0,\dots,0)\}$, $\label{7} \Gamma(iY^{\beta})=\displaystyle{\sum_{\gamma \le \beta} \frac{(-1)^{|\gamma|}}{\gamma !}Y_*^{\gamma}\left (Y^{\beta-\gamma}-\frac{1}{(\beta-\gamma)!}Y_*^{\beta-\gamma}\right )}. $
\item For every $\beta \in A^{\cal I}$,
 $\label{8} \displaystyle{Y^{\beta}-\frac{1}{\beta!}Y_*^{\beta}=Y_*^{\beta}(Y^{(0,\dots,0)}-i)+\sum_{\gamma \le \beta,\gamma \ne(0,\dots,0)}\frac{1}{(\beta-\gamma)!}Y_*^{\beta-\gamma}\Gamma(iY^{\gamma})}.$
\end{enumerate}
\end{lem}

\begin{proof}
\begin{enumerate}[nolistsep]
\item Let $\beta \in A^{\cal I}$. If $0 \le k \le \beta_j$ then  $ad^k X_j(Y^{\beta})=Y^{\beta-k\delta^j}$, if  $k> \beta_i$ then $ad^k X_i(Y^{\beta})=0$, which proves Formula $\eqref{eq1gamma}$. Let $(j,l)\in \llbracket 1,n \rrbracket^2$ and $k\in \mathbb{N}$, if $j\ne l$ then $ad X_j((Y^{\delta^j})^k)=0 $. Using   for each $j\in \llbracket 1,n \rrbracket$  the previous formula we obtain
\begin{equation*}
\begin{split}
\Gamma( iY^{\beta})&=\sum_{(\gamma_1,\gamma_2,\dots,\gamma_n)\in \prod_{j=1}^n  \llbracket 0,\beta_j \rrbracket} i\frac{i^{\gamma_n}}{\gamma_n!}\frac{i^{\gamma_{n-1}}}{\gamma_{n-1}!}\dots\frac{i^{\gamma_1}}{\gamma_1!}(Y^{\delta^n})^{\gamma_n}(Y^{\delta^{n-1}})^{\gamma_{n-1}}\dots (Y^{\delta^1})^{\gamma_1}Y^{\beta-\sum_{i=1}^n\gamma_i\delta^i}\\[-6pt]
&= \sum_{\gamma \le \beta} \frac{(-1)^{|\gamma|}}{\gamma !}Y_*^{\gamma}Y^{\beta-\gamma}.
\end{split}
\end{equation*}
\item  Let $\beta \in A^{\cal I}\setminus \{(0,\dots,0)\}$. From previous formula we have

$\Gamma(iY^{\beta})=\displaystyle{\sum_{\gamma \le \beta} \frac{(-1)^{|\gamma|}}{\gamma!}Y_*^{\gamma}\left (Y^{\beta-\gamma}-\frac{1}{(\beta-\gamma)!}Y_*^{\beta-\gamma}\right )+ \frac{iY_*^{\beta}}{\beta!}\left (\sum_{\gamma \le \beta}(-1)^{|\gamma|}\binom{\beta}{\gamma}\right )}$,

 and $\displaystyle{\sum_{\gamma \le \beta}(-1)^{|\gamma|}\binom{\beta}{\gamma}=0}$ thus the formula holds.

\item Let $\beta \in A^{\cal I}$. We set $Z_{(0,\dots,0)}=Y^{(0,\dots,0)}-i$
   and $Z_{\beta}=\beta! \Gamma(iY^{\beta})=\sum_{\gamma \le \beta}(-1)^{|\gamma|}\binom{\beta}{\gamma} Y_*^{\gamma}\left ((\beta-\gamma)!Y^{\beta-\gamma}-Y_*^{\beta-\gamma} \right)$.
 Pascal's binomial inversion formula  states that  $\beta!Y^{\beta}-Y_*^{\beta}=\sum_{\gamma \le \beta}\binom{\beta}{\gamma}Y_*^{\beta-\gamma}Z_{\gamma}$.
\end{enumerate}
\end{proof}
\noindent With Lemma $\ref{lemgenIf}$ and formulas  $\eqref{7}$ et $\eqref{8}$, we can easily prove next Corollary.

\begin{cor}$\label{corgenIf}$ We have  $I(f)=\langle Y_*^{\beta}-\beta!Y^{\beta}|\beta \in A^{\cal I} \rangle_{ideal}$.
\end{cor}
\noindent From now on, for every $(U,V) \in U(\lig_{\alpha,{\cal I}})^2$, we shall denote
$$U\equiv V~~mod(I(f)) ~~~\hbox{ if and only if}~~~  U-V \in I(f).$$

\begin{lem}$\label{lemTIf}$

\begin{enumerate}[nolistsep]
\item We have $Y^{(0,\dots,0)} \in T(I(f))$.
\item Let $\beta, \gamma \in A^{\cal I}$. If there exists $i\in \llbracket 1,n \rrbracket$  $1\le \beta_i < \alpha_i$ and $1\le \gamma_i < \alpha_i$ then $Y^{\beta}Y^{\gamma} \in T(I(f))$
\end{enumerate}

\end{lem}
\begin{proof}
\begin{enumerate}[nolistsep]
\item As $Y^{(0,\dots,0)} -i\in I(f)$ then $Y^{(0,\dots,0)} \in T(I(f))$.
\item Let $\beta, \gamma \in A^{\cal I}$ and $i\in \llbracket 1,n \rrbracket$ such that $1\le \beta_i < \alpha_i$ et $1\le \gamma_i < \alpha_i$.
The computation
\begin{equation*}
\begin{split}
Y^{\beta}Y^{\gamma} -\frac{\gamma_i +1}{\beta_i}Y^{\beta-\delta^i}Y^{\gamma+\delta_i}
 &\equiv\frac{1}{\beta!}Y_*^{\beta}\frac{1}{\gamma!}Y_*^{\gamma}-\frac{\gamma_i +1}{\beta_i}\frac{1}{(\beta-\delta^i)!}Y_*^{\beta-\delta^i}\frac{1}{(\gamma+\delta^i)!}Y_*^{\gamma+\delta^i}~~mod(I(f)) \\[-6pt] &\equiv0~~mod(I(f))
\end{split}
\end{equation*}
 proves that $Y^{\beta}Y^{\gamma} \in T(I(f))$ since $Y^{\beta-\delta^i}Y^{\gamma+\delta_i}<Y^{\beta}Y^{\gamma}$.
\end{enumerate}
\end{proof}

\begin{lem}$\label{lemcroc}$ Let $T=X_1^{p_1}\dots X_n^{p_n} \prod_{\beta \in A^{\cal I}}(Y^{\beta})^{q_{\beta}}\in {\cal S}$. Then,
\begin{enumerate}[nolistsep]
\item For every $k \in \llbracket 1,n \rrbracket$,  $X_k[T,Y^{\delta^k}]\equiv p_kT~~mod(I(f))$.
\item For every $k\in \llbracket 1,n \rrbracket$,  $[X_k,T]Y^{\delta^i}\equiv (\sum_{\beta \in A^{\cal I}} q_{\beta}\beta_k)T~~mod(I(f))$.
\end{enumerate}
\end{lem}
\begin{proof}

\begin{enumerate}
\item  Let $k\in \llbracket 1,n \rrbracket$,
\begin{equation*}
\begin{split}
X_k[T,Y^{\delta^k}]&=X_k\left(p_kX_1^{p_1}\dots X_k^{p_k-1}\dots X_n^{p_n} \prod_{\beta \in A^{\cal I}}(Y^{\beta})^{q_{\beta}}Y^{(0,\dots,0)} \right)\\[-6pt]
&=p_kT+ip_kX_1^{p_1}\dots X_k^{p_k}\dots X_n^{p_n} \prod_{\beta \in A^{\cal I}}\left(Y^{\beta})^{q_{\beta}}(Y^{(0,\dots,0)}-i \right),
\end{split}
\end{equation*}
as $Y^{(0,\dots,0)}-i \in I(f)$ we can conclude.

\item Let $k\in \llbracket 1,n \rrbracket$, we have

\begin{equation*}
\begin{split}
[X_k,T]Y^{\delta^k}&=X_1^{p_1}\dots X_n^{p_n}\sum_{\beta \in A^{\cal I},\beta_k \ge1} q_{\beta} \left (\prod_{\beta' \in A^{\cal I},\beta'\ne \beta}(Y^{\beta'})^{q_{\beta'}}\right )Y^{\beta-\delta^k}Y^{\delta^k}\\[-6pt]
                   &\equiv\left(\sum_{\beta \in A^{\cal I}} q_{\beta}\beta_k\right)T~~mod(I(f)),
\end{split}
\end{equation*}
since $Y^{\beta-\delta^k}Y^{\delta^k}-i\beta_kY^{\beta}$ is in $I(f)$.
\end{enumerate}
\end{proof}

\subsection{ Reduction in the algebra  $U(\lig_{\alpha,{\cal I}})$}\label{sec4}
For every $a=(a_k)_{1\le k \le n} \in \mathbb{C}^n$ and $b=(b_k)_{1\le k \le n} \in \mathbb{C}^n$ , we define the operators:

 \begin{equation*} \begin{application}{G_{a,b}}{U(\lig_{\alpha,{\cal I}})}{U(\lig_{\alpha,{\cal I}})}{T}{\displaystyle{\sum_{k=1}^na_kX_k[T,Y^{\delta^k}]+b_k[X_k,T]Y^{\delta^k}}.} \end{application} \end{equation*}
 Every two-sided ideal of $U(\lig_{\alpha,{\cal I}})$, in particular the ideal $I(f)$,  is  invariant under the operators $G_{a,b}$.

\begin{lem}$\label{redOif}$  Let $T=X_1^{p_1}\dots X_n^{p_n} \prod_{\beta \in A^{\cal I}}(Y^{\beta})^{q_{\beta}}\in O(I(f))$. There exists a $p$-uple $\left((i_1,r_{i_1})\dots,(i_p,r_{i_p})\right )\in \prod_{j=1}^p (I_j \times \llbracket 1,\alpha_{i_j} \rrbracket)$ such that:
\begin{equation*}
G_{a,b}(T)\equiv\left(s+\sum_{j=1}^p (\frac{r_{i_j}}{\alpha_{i_j}}-1) \right)T~~mod(I(f)),
\end{equation*}
where  $a=(1,\dots,1)$, $b=\sum_{j=1}^p \frac{1}{\alpha_{i_j}}\delta^{i_j}$ and $s=\deg(T)$.
\end{lem}

\begin{proof} Let $T=X_1^{p_1}\dots X_n^{p_n} \prod_{\beta \in A^{\cal I}}(Y^{\beta})^{q_{\beta}}\in O(I(f))$ and $s=\deg(T)$. Let $j \in \llbracket 1,p \rrbracket$, we will construct
$(i_j,r_{i_j})\in I_j \times \llbracket 1,\alpha_{i_j} \rrbracket$ by separation of cases.\\

\noindent Case 1: There exists  $\beta \in \alpha^j$ such that $q_{\beta} \ne 0$. We denote $\beta$ the minimum of these $n$-uples. As $T \in O(I(f))$, from Lemmas $\ref{lemsecfini}$ and $\ref{lemTIf}$ then
$\beta \ne(0,\dots,0)$. So there exists $i_j \in I_j$ such that $\beta_{i_j} \ne 0$. For every $\beta' \in \alpha^j$ such that $q_{\beta'} \ne 0$,  then $\beta_{i_j}'=\alpha_{i_j}$
 as consequence of Lemma $\ref{lemTIf}$. From Lemma $\ref{lemcroc}$,  $[X_{i_j},T]Y^{\delta^{i_j}}\equiv q_{\beta}\beta_{i_j}T+(\sum_{\beta' \in \alpha^j, \beta' \ne \beta} q_{\beta'}\alpha_{i_j}) T~~mod(I(f))$. Setting $r_{i_j}=\beta_{i_j}\in \llbracket 1,\alpha_{i_j} \rrbracket $, the formula:
 $$\frac{1}{\alpha_{i_j}}[X_{i_j},T]Y^{\delta^{i_j}}\equiv(\sum_{\beta' \in \alpha^j } q_{\beta'}+(\frac{r_{i_j}}{\alpha_{i_j}}-1)) T~~mod(I(f))$$
is true if $r_{i_j}=\alpha_{i_j}$, and also if $1\le r_{i_j}<\alpha_{i_j}$ since in this case $q_{\beta}=1$ from Lemma $\ref{lemTIf}$.\\

\noindent Case 2: For every  $\beta \in \alpha^j$, we have  $q_{\beta} = 0$.  We can make an arbitrary choice of $i_j$ in  $I_j$,  and then put $r_{i_j}=\alpha_{i_j}$.
The equality $\frac{1}{\alpha_{i_j}}[X_{i_j},T]Y^{\delta^{i_j}}\equiv\left(\sum_{\beta' \in \alpha^j }q_{\beta '}+(\frac{r_{i_j}}{\alpha_{i_j}}-1)\right) T~~mod(I(f))$ holds because the  members on each side vanish. Summing up we obtain
$$\sum_{j=1}^p \left (\frac{1}{\alpha_{i_j}}[X_{i_j},T]Y^{\delta^{i_j}} \right )\equiv\left( \sum_{j=1}^p ((\sum_{\beta' \in \alpha^j } q_{\beta'}T+\sum_{j=1}^p(\frac{\beta_{i_j}}{r_{i_j}}-1))\right )T~~mod(I(f)).$$

The family $(\alpha^j)_{1 \le j \le p}$ is a partition of $A^{\cal I}$ then $\sum_{j=1}^p \left(\sum_{\beta' \in \alpha^j } q_{\beta '} \right) =\sum_{\beta \in A^{\cal I}}q_{\beta}$. Moreover  $\sum_{i=1}^n X_i[T,Y^{\delta^i}]\equiv(\sum_{i=1}^n p_i) T~~mod(I(f))$ from Lemma $\ref{lemcroc}$. By definition $s= \sum_{i=1}^n p_i+\sum_{\beta \in A^{\cal I}}q_{\beta}$, hence
$$G_{(1,\dots,1),\sum_{j=1}^p \frac{1}{\alpha_{i_j}}\delta^{i_j}}(T)\equiv\left(s+\sum_{j=1}^p (\frac{r_{i_j}}{\alpha_{i_j}}-1)\right)T~~mod(I(f)).$$
\end{proof}

\begin{pr} $\label{Ld}$ Let $s \in \mathbb{N}$. We define the operator:

\begin{equation*} G_s=\prod_{i=(i_1,\dots,i_p)\in \prod_{j=1}^p I_j} \left( \prod_{(r_{i_1},\dots,r_{i_j})\in \prod_{j=1}^p \llbracket 1,\alpha_{i_j} \rrbracket} G_{a,b_i}-\left(s+\sum_{j=1}^p (\frac{r_{i_j}}{\alpha_{i_j}}-1)\right) \right),\end{equation*}

with $a=(1,\dots,1)$ and $b_i=\sum_{j=1}^p \frac{1}{\alpha_{i_j}}\delta^{i_j}$. For every $T=X_1^{p_1}\dots X_n^{p_n} \prod_{\beta \in A^{\cal I}}(Y^{\beta})^{q_{\beta}}\in O(I(f))$ such that  $\deg(T)=s $ we have
  $$G_s(T)\equiv 0~~mod(I(f)).$$
\end{pr}
\begin{proof} Let $T=X_1^{p_1}\dots X_n^{p_n} \prod_{\beta \in A^{\cal I}}(Y^{\beta})^{q_{\beta}}\in O(I(f))$, denote $s=\deg(T)$. To compute $G_s(T)$ modulo $I(f)$ we  commute the factors (see Lemma $\ref{redOif}$) in order to place in first position one of them which annihilates $T$ modulo $I(f)$ (see Lemma $\ref{redOif}$). Since the ideal $I(f)$ is invariant for each factor of $G_s$ we have $G_s(T)\equiv 0~~mod(I(f))$.
\end{proof}
\noindent We now define the operators:
$$\label{hab} \begin{application}{H_{a,b}}{U(\lig_{\alpha,{\cal I}})}{U(\lig_{\alpha,{\cal I}})}{T}{\displaystyle{\sum_{k=1}^na_k[X_kT,Y^{\delta^k}]+b_k[X_k,TY^{\delta^k}].}} \end{application}$$
For $a\in \mathbb{C}^n$, we set $|a|=\sum_{k=1}^n a_k$.
\begin{lem} For every $T$ in $U(\lig_{\alpha,{\cal I}})$ we have
$H_{a,b}(T)\equiv G_{a,b}(T)+(|a|+|b|)T~~mod(I(f))$.
\end{lem}
\begin{proof} Let $T\in U(\lig_{\alpha,{\cal I}})$,
\begin{equation*}
\begin{split}
H_{a,b}(T)&=G_{a,b}(T)+\displaystyle{\sum_{k=1}^na_k[T_k,Y^{\delta^k}]T+b_kT[X_k,Y^{\delta^k}]}\\[-6pt]
& =G_{a,b}(T)+(|a|+|b|)T+(|a|+|b|)(Y^{(0,\dots,0)}-1)T\\[-6pt]
&\equiv G_{a,b}(T)+(|a|+|b|)T~~mod(I(f)).
\end{split}
\end{equation*}
\end{proof}
\noindent We reformulate the Proposition $\ref{Ld}$ with  operators $H_s$ instead of $G_s$.
\begin{pr} $\label{hs}$ Let $s \in \mathbb{N}$. We define the operator

\begin{equation*} H_s=\prod_{i=(i_1,\dots,i_p)\in \prod_{j=1}^p I_j} \left ( \prod_{(r_{i_1},\dots,r_{i_j})\in \prod_{j=1}^p \llbracket 1,\alpha_{i_j} \rrbracket} H_{a,b_i}-(s+n-p+\sum_{j=1}^p \frac{r_{i_j}+1}{\alpha_{i_j}}) \right),\end{equation*}

with $a=(1,\dots,1)$ et $b_i=\sum_{j=1}^p \frac{1}{\alpha_{i_j}}\delta^{i_j}$.

For every $T=X_1^{p_1}\dots X_n^{p_n} \prod_{\beta \in A^{\cal I}}(Y^{\beta})^{q_{\beta}}\in {\cal S}$ such that  $s=\deg(T)$ we have  $$H_s(T) \equiv 0~~mod(I(f)).$$
\end{pr}

The next proposition tells us how the operators $H_s$ operators reduce the degree in $U(\lig_{\alpha,{\cal I}})$. We set $U_{-1}(\lig_{\alpha,{\cal I}})=\{0\}$ by convention.

\begin{pr} $\label{redHd}$  For every integer $s$, $H_s(U_{s}(\lig_{\alpha,{\cal I}})) \subset U_{s-1}(\lig_{\alpha,{\cal I}})+I(f)$.
\end{pr}
\begin{proof} Let $s \in \mathbb{N}$ and $U \in U_{s}(\lig_{\alpha,{\cal I}})$. There exists one and only one decomposition $U=\Phi+Can(U,I(f))$ with $\Phi \in I(f)$ and $Can(U,I(f))\in Vect(O(I(f))$, see Proposition $\ref{pr:OI}$. From Proposition $\ref{pr:OImin}$,
$T(Can(U,I(f)) \preceq T(U)$ then $Can(U,I(f)) \in U_{s}(\lig_{\alpha,{\cal I}})$. Let $Can(U,I(f))=V_s+V_{s-1}$, where $V_s$ is the homogeneous part of degree $s$ of $Can(U,I(f))$. By linearity $H_s(U)=H_s(\Phi)+H_s(V_s)+H_s(V_{s-1})$, $H_s(\Phi) \in I(f)$  since every two-sided ideal of $U(\lig_{\alpha,{\cal I}})$ are invariant for $H_s$, $H_s(V_s) \in I(f)$ from Proposition $\ref{hs}$ and $H_s(V_{s-1})\in U_{s-1}(\lig_{\alpha,{\cal I}})+I(f)$ from Lemma $\ref{lemcroc}$. Hence $H_s(U)\in U_{s-1}(\lig_{\alpha,{\cal I}})+I(f)$.

\end{proof}

 \subsection{Meromorphic continuation}

$\label{sec5}$
We denote by ${\cal D}(\lig_{\alpha,{\cal I}})$ the algebra $\rho(U(\lig_{\alpha,{\cal I}}))$ and for every  $q\in \mathbb{N}$, ${\cal D}(\lig_{\alpha,{\cal I}})^q=\rho({\cal U}_q(\lig_{\alpha,{\cal I}}))$. Let $\Delta_1$ the Goodman laplacian defined as
$\Delta_1=\rho(1+\Delta_G)$ where  $\Delta_G=-\left (\sum_{k=1}^n X_k^2+\sum_{\beta \in A^{\cal I}}(Y^{\beta})^2 \right) $ (see \cite{Goodman}).

\begin{pr} The triple $({\cal D}(\lig_{\alpha,{\cal I}}),\Delta_1,2)$ is an algebra  of generalized differential operators, and then $\Delta_1$ is a generalized laplacian of order 2  (see Definition $\ref{algdif}$), so that:
\begin{enumerate}[nolistsep]
\item  $({\cal D}(\lig_{\alpha,{\cal I}})^q=\rho({\cal U}_q(\lig_{\alpha,{\cal I}})))_{q\in \mathbb{N}}$ is an increasing filtration of ${\cal D}(\lig_{\alpha,{\cal I}})$.
\item The unity is of order 0.
\item For every $\tilde{X}\in {\cal D}(\lig_{\alpha,{\cal I}})^{q_{\tilde{X}}}$, $[\Delta_1,\tilde{X}]\in {\cal D}(\lig_{\alpha,{\cal I}})^{q_{\tilde{X}}+1}$.
\item $\label {ellip}$ For every  $\tilde{X}\in {\cal D}(\lig_{\alpha,{\cal I}})^{q_{\tilde{X}}}$, there exists $\varepsilon>0 $ such that:
 \begin{equation*} \forall v \in \Hi^{\infty},~~\|\Delta_1^{\frac{q}{2}}v \|+\|v\|\ge \varepsilon \|\tilde{X}v\|,~~~\text{~(generalized G{\aa}rding inequality).}\end{equation*}
\end{enumerate}
\end{pr}
\begin{proof}
 The filtration on $({\cal D}(\lig_{\alpha,{\cal I}})$ is  transported from the standard degree filtration on $U(\lig_{\alpha,{\cal I}})$ by the algebra morphism $\rho$, so the two first items are clear. The third one also, in light of $\deg(\Delta_G)=2$. The generalized G{\aa}rding inequality, comes from Proposition \cite[Proposition 2.1]{Goodman}.
\end{proof}

\noindent For a proof of the next proposition, see \cite[proposition I.1]{Domi5} and \cite[Proposition IV.1]{Domi5}.

\begin{pr}\label{Deltaschatten}
There exists some real $s_0 \ge 1$ such that $\Delta_1^{-\frac{1}{2}} \in {\cal L}^{s_0}(\Hi)$.
\end{pr}
Let $(a,b)\in (\mathbb{C}^n)^2$ and $s\in \mathbb{N}$.

  To the endomorphisms  $H_{a,b}$ and $H_s$ on  $U(\lig_{\alpha,{\cal I}})$ (see $\eqref{hab}$ and $\eqref{hs}$) ,  we associate  the endomorphisms $\widetilde{H}_{a,b}$ and $T_s$ on  ${\cal D}(\lig_{\alpha,{\cal I}})$:

\begin{equation*}\label{TSLIE}
\begin{split}
&\begin{application}{\widetilde{H}_{a,b}}{{\cal D}(\lig_{\alpha,{\cal I}})}{{\cal D}(\lig_{\alpha,{\cal I}})}{\widetilde{T}}{\displaystyle{\sum_{k=1}^na_k[\rho(X_k)\widetilde{T},\rho{Y^{\delta^k}}]+b_k[\rho{X_k},\widetilde{T}\rho(Y^{\delta^k})],}} \end{application}\\[-6pt]
&T_s=\prod_{(i_1,\dots,i_p)\in \prod_{j=1}^p I_j} \left ( \prod_{(r_{i_1},\dots,r{i_p})\in \prod_{j=1}^p \llbracket 1,\alpha_{i_j} \rrbracket} \tilde{H}_{(1,\dots,1),\sum_{j=1}^p \frac{1}{\alpha_{i_j}}\delta^{i_j}}-(s+n-p+\sum_{j=1}^p \frac{r_{i_j}+1}{\alpha_{i_j}}) \right).
\end{split}
\end{equation*}
As $\rho$ is an algebra morphism, this commutative diagram holds:
$$
\xymatrix{
    U(\lig_{\alpha,{\cal I}}) \ar[r]^{H_s} \ar[d]_{\rho}  & U(\lig_{\alpha,{\cal I}}) \ar[d]^{\rho} \\
    {\cal D}(\lig_{\alpha,{\cal I}}) \ar[r]^{T_s} & {\cal D}(\lig_{\alpha,{\cal I}}).
  }
$$
We are now able to establish the meromorphic continuation theorem for the Lie algebras $\lig_{\alpha,{\cal I}}$.

\begin{thm} $\label{th:merolig}$ Let be a Lie algebra  $\lig_{\alpha,{\cal I}}$. Let a real $s_0\ge 1$ such that $\Delta_1^{-\frac{1}{2}} \in {\cal L}^{s_0}(\Hi)$.
For every $\tilde{X} \in {\cal D}(\lig_{\alpha,{\cal I}})^{q_{\tilde{X}}}$,  the spectral zeta function:
\begin{equation*} \application{\zeta_{\tilde{X},\Delta_1}}{\mathbb{C}_{<-\frac{s_0+q}{2}}}{\mathbb{C}}{z}{{\rm Trace}(\tilde{X}\Delta_1^{z})}. \end{equation*}

is holomorphic and admits a meromorphic continuation to all of $\mathbb{C}$. The poles are constrained in the family $\Omega_{\alpha,{\cal I}}=(\omega_{i,r,l})$ where

$$\omega_{i,r,l}=\frac{p-n-q}{2}-\sum_{j=1}^p \frac{r_{i_j}+1}{2\alpha_{i_j}}+\frac{l}{2}, $$

for every $i=(i_1,\dots,i_p)\in \prod_{j=1}^p I_j$ , $r=(r_{i_1},\dots,r_{i_j})\in \prod_{j=1}^p \llbracket 1,\alpha_{i_j} \rrbracket$ and $l \in \mathbb{Z}$ such that $l \ge  \sum_{j=1}^p \frac{r_{i_j}+1}{\alpha_{i_j}}+n-p-s_0$. The multiplicity order of each pole is equal to or less  the number of times it appears in the family $\Omega_{\alpha,{\cal I}}$.

\end{thm}

\begin{proof}Let $(a,b)\in (\mathbb{C}^n)^2$. We prove that $(T_s)_{s \in \mathbb{N}}$ is a  reduction sequence in the algebra ${\cal D}(\lig_{\alpha,{\cal I}})$ (see Definition  $\ref{suitered}$). For every $k\in \llbracket 1,n \rrbracket$, denote $P_k=\rho(X_k)$ et $Q_k=\rho(Y^{\delta^k})$. As  $[P_k,Q_k]=1$ every pairs $\left ((P_k)_{1\le k\le n},(Q_k)_{1\le k\le n}\right )$  and $(a,b)\in (\mathbb{C}^n)^2$ are always compatible. Here $\Delta_1$ is already in ${\cal D}(\lig_{\alpha,{\cal I}})$. For every $s\in \mathbb{N}$ and every $\tilde{X} \in {\cal D}(\lig_{\alpha,{\cal I}})^s$ we have just to prove that  $T_s(\tilde{X}) \in {\cal D}(\lig_{\alpha,{\cal I}})^{s-1}$.\\

Let $\tilde{X} \in {\cal D}(\lig_{\alpha,{\cal I}})^s$ and $X \in U_s(\lig_{\alpha,{\cal I}})$ such that $\rho(X)=\tilde{X}$. From previous commutative diagram, we stand $T_s(\tilde{X})=\rho(H_s(X))$. From Proposition $\ref{redHd}$, $H_s(X) \in  U_{s-1}(\lig_{\alpha,{\cal I}})+I(f)$ then $ T_s(\tilde{X})\in \rho(U_{s-1}(\lig_{\alpha,{\cal I}})$ since $I(f)$ is the kernel of $\rho$. Hence, we have proved $T_s(\tilde{X}) \in {\cal D}(\lig_{\alpha,{\cal I}})^{s-1}$. The polynomial function associates to the reduction sequence $(T_s)_{s \in \mathbb{N}}$ is defined by:
 \begin{equation*}
 \forall z' \in \mathbb{C},~ b(z')=\prod_{(i_1,\dots,i_p)\in \prod_{j=1}^p I_j} \left ( \prod_{(r_{i_1},\dots,r_{i_p})\in \prod_{j=1}^p \llbracket 1,\alpha_{i_j} \rrbracket} (p-n-\sum_{j=1}^p \frac{r_{i_j}+1}{\alpha_{i_j}}-z') \right).
 \end{equation*}
We conclude with  Theorem \ref{thetheorem}.
\end{proof}
\nocite{Avoros1} \nocite{Dixmier2} \nocite{Kirillovbook} \nocite{Rudin} \nocite{Cimpa}

\bibliographystyle{plain}
\bibliography{bibliofranck}

\end{document}